\documentclass [leqno, 10pt]{amsart}

\usepackage{etex} 
\usepackage{amssymb, amsmath, amsfonts, amsthm, amsbsy, amscd}
\usepackage[bbgreekl]{mathbbol} 
\usepackage[mathscr]{euscript}
\usepackage{tabularx, caption}
\usepackage{enumerate}

\usepackage{url}
\usepackage[linktocpage]{hyperref}

\usepackage{filecontents}
\usepackage{multicol}
\usepackage{graphicx}
\usepackage{float}
\usepackage{subfig}
\usepackage{wrapfig}
\usepackage{empheq}
\usepackage{fancybox}

\usepackage{etoolbox}
\apptocmd{\sloppy}{\hbadness 10000\relax}{}{}

\newtheorem{theorem}{Theorem}[section]
\newtheorem*{theorem*}{Theorem}{\,}
\newtheorem{corollary}[theorem]{Corollary}
\newtheorem{lemma}[theorem]{Lemma}
\newtheorem{definition}{Definition}[section]
\newtheorem*{definition*}{Definition}{}

\usepackage{thmtools}
\theoremstyle{definition}
\declaretheorem[style=definition,qed=$\Diamond$,sibling=definition]{example}
\declaretheorem[style=definition,qed=$\Diamond$,sibling=definition]{remark}

\numberwithin{equation}{section}

\newcommand{\one}{\mathbb{1}}
\newcommand{\C}{\mathscr{C}}
\newcommand{\sm}{\mathscr{S}}
\newcommand{\uinf}{\mathscr{U}}

\newcommand{\sphere}{\mathbb{S}}
\newcommand{\sech}{\operatorname{sech}}
\newcommand{\stl}{\mathbb{L}}
\newcommand{\gau}{\mathscr{K}}

\newcommand{\muf}{\mu}

\newcommand{\amc}{\mathcal{H}}

\newcommand{\gn}{N}
\newcommand{\std}{\mathbb{V}^{\ast}}

\newcommand{\Om}{\Omega}

\newcommand{\stw}{\mathbb{W}}

\newcommand{\imt}{\iota}

\newcommand{\A}{\mathscr{A}}
\newcommand{\U}{\mathsf{U}}

\newcommand{\lc}{\Sigma}

\renewcommand{\H}{\mathsf{H}}

\newcommand{\pol}{\mathsf{Pol}}

\newcommand{\om}{\omega}

\newcommand{\hess}{\operatorname{Hess}}
\newcommand{\adj}{\operatorname{adj}}
\newcommand{\vol}{\mathsf{vol}}

\newcommand{\ka}{\kappa}

\newcommand{\La}{\Lambda}
\newcommand{\T}{\mathcal{T}}

\renewcommand{\part}{\vdash}

\newcommand{\sff}{\Pi}

\newcommand{\dum}{\,\cdot\,\,}
\newcommand{\Ga}{\Gamma}

\newcommand{\nm}{\mathsf{W}}

\newcommand{\la}{\lambda}
\newcommand{\ep}{\epsilon}

\newcommand{\reat}{\mathbb{R}^{\times}}

\newcommand{\cinf}{C^{\infty}}

\newcommand{\eno}{\text{End}}

\newcommand{\si}{\sigma}
\newcommand{\pr}{\partial}

\newcommand{\sign}{\operatorname{sgn}}
\newcommand{\bnabla}{\bar{\nabla}}

\newcommand{\en}{[\nabla]}
\newcommand{\lie}{\mathfrak{L}}

\newcommand{\Aff}{\mathbb{Aff}}

\newcommand{\lb}{\langle}
\newcommand{\ra}{\rangle}

\newcommand{\ste}{\mathbb{V}}

\newcommand{\al}{\alpha}
\newcommand{\be}{\beta}
\newcommand{\ga}{\gamma}

\newcommand{\hnabla}{\widehat{\nabla}}
\newcommand{\tnabla}{\tilde{\nabla}}

\newcommand{\eul}{\mathbb{X}}

\newcommand{\tensor}{\otimes}

\newcommand{\rea}{\mathbb R}

\newcommand{\tr}{\operatorname{\mathsf{tr}}}

\newcommand{\ann}{\text{Ann}\,}

\renewcommand{\L}{\mathbb{L}}

\begin{document}
\title[Equiaffine geometry of level sets]{Equiaffine geometry of level sets and ruled hypersurfaces with equiaffine mean curvature zero}
\author{Daniel J.~F. Fox} 
\address{Escuela T\'ecnica Superior de Ingenier\'ia y Dise\~no Industrial\\ Universidad Polit\'ecnica de Madrid\\Ronda de Valencia 3\\ 28012 Madrid Espa\~na}
\email{daniel.fox@upm.es}


\begin{abstract}
Basic aspects of the equiaffine geometry of level sets are developed systematically. 
As an application there are constructed families of $2n$-dimensional nondegenerate hypersurfaces ruled by $n$-planes, having equiaffine mean curvature zero, and solving the affine normal flow. Each carries a symplectic structure with respect to which the ruling is Lagrangian. 
\end{abstract}

\maketitle

\section{Introduction}
The goal of this article is the construction of $2n$-dimensional equiaffine mean curvature hypersurfaces ruled by totally geodesic $n$-planes. To achieve this aim, basic aspects of the equiaffine geometry of level sets are developed systematically. In the process there are made some remarks about the equiaffine geometry of level sets and the construction of the affine normal distribution in spaces more general than flat affine space. 

The examples constructed generalize to higher dimensions equiaffine mean curvature zero ruled surfaces constructed by A. Mart\'inez and F. Milan in \cite{Martinez-Milan}. Precisely, Theorem \ref{lifttheorem} shows that, given $Q \in \cinf(M)$ and a centroaffine immersion $A:M \to \rea^{n+1}$ such that the pullback of the contraction of the radial Euler field with the volume form on $\rea^{n+1}$ is a nonzero constant multiple of the volume form $du^{1} \wedge \dots \wedge du^{n}$ for some flat affine coordinates on $M$, then, with respect to an appropriate equiaffine structure on $M \times \rea^{n+1\,\ast}$, the level sets of the function $F:M \times \rea^{n+1\,\ast} \to \rea$ defined by $F(u, x) = \lb A(u), x\ra + Q(u)$ are equiaffine mean curvature zero hypersurfaces foliated by $n$-dimensional affine planes isotropic with respect to the equiaffine metric. Moreover each level set carries a symplectic structure with respect to which the leaves of the ruling are Lagrangian. By Lemma \ref{affinespherelemma}, in general a level set of $F$ is not an improper affine sphere, for this is the case if and only if the image of $A$ is contained in a hyperplane. Finally, there is an explicit smooth map $\varphi:\rea \times \rea^{2n} \to M \times \rea^{n+1\,}$ such that $F(\varphi(t, p))$ is a constant multiple of $t$, so that $\varphi(t,\dum)$ parameterizes part of a level set of $F$, and that solves the affine normal flow in the sense that $\tfrac{d}{dt}\varphi(t, p)$ is equal to the affine normal $\nm_{\varphi(t, p)}$ at $p$ to the image $\varphi(t,\rea^{2n})$.

In the case $n = 2$, taking $Q = 0$ and taking as the components of $A$ linearly independent solutions of a homogeneous linear second order differential equation yields examples in \cite{Martinez-Milan}. The prototypical example is a level set $\Sigma_{t} = \{(u, x, y): F(u,x , y) = t\}$ of $F(u,x,y) = x \sin u + y \cos u$. The equiaffine normal is $-\sin u \pr_{x} - \cos u \pr_{y}$, the symplectic form is the restriction of $du \wedge (\cos u dx - \sin u dy)$, and the ruling is generated by $\cos u \pr_{x} - \sin u \pr_{y}$. The level set $\Sigma_{t}$ is a helicoid with the parameterization $\{(r, s\sin r + t\cos r, -s\cos r + t \sin r):(r, s) \in \rea^{2}\}$.

In section \ref{zamcsection}, several recipes for explicitly constructing centroaffine immersions with the necessary properties are given. A typical example obtained from these constructions is that, for any smooth function $Q(u, v)$ of two variables, the level sets in $\rea^{5}$ of the function
\begin{align}\label{genhel}
\begin{split}
F(&x_{1}, x_{2}, x_{3}, x_{4}, x_{5}) \\
&= x_{3}\sech{x_{1}}\cos( x_{2}\cosh^{2}{x_{1}}) + x_{4} \sech{x_{1}} \sin( x_{2}\cosh^{2}{x_{1}}) + x_{5}\tanh{x_{1}} + Q(x_{1}, x_{2})
\end{split}
\end{align} 
are smoothly immersed equiaffine mean curvature zero hypersurfaces ruled by $2$-planes. 

In flat affine space any two parallel metrics of the same signature are affinely equivalent. While the Gauss-Kronecker curvature of a hypersurface with respect to different parallel metrics changes, the condition that it vanish or not does not; a hypersurface is nondegenerate if its Gauss-Kronecker curvature is everywhere nonzero with respect to any parallel metric (equivalently, the second fundamental form is nondegenerate), and similarly, it makes sense to speak of a hypersurface with zero Gauss-Kronecker curvature with respect to any parallel metric (a \textit{developable} hypersurface). A \textit{ruling} of a hypersurface by $k$-planes is a foliation by totally geodesic $k$-dimensional submanifolds, and a hypersurface equipped with a ruling is said to be \textit{ruled}. A ruling is \textit{cylindrical} if its leaves are parallel planes. A cylindrical ruling is developable because the vectors tangent to the ruling lie in the radical of the second fundamental form. However, there exist developable hypersurfaces that are not cylindrical. 
By construction the level sets of the the functions $F$ of Section \ref{examplesection} are ruled but not developable. On the other hand, the graph of $F$ is a smoothly immersed noncylindrical developable (Gauss-Kronecker curvature zero) hypersurface.

Because the plane tangent to a ruling is isotropic with respect to the second fundamental form, the second fundamental form of a ruled hypersurface has indefinite signature. If the hypersurface is nondegenerate then the rank of a ruling can be no larger than the maximum possible dimension of an isotropic subspace. For a $2n$-dimensional hypersurface, the maximum possible dimension of an isotropic subspace is $n$, and so the maximal rank of a ruling is $n$, and a nondegenerate hypersurface admitting such a ruling necessarily has split signature second fundamental form. The examples constructed here are maximally ruled in the sense that they carry rulings with the maximal possible rank. It would be interesting to know if there are maximally ruled equiaffine mean curvature hypersurfaces that are not equivalent to these examples.

The hypersurfaces constructed are realized as level sets, and so it is necessary to record, in Section \ref{prelimsection}, some facts and formulas related to the equiaffine geometry of level sets that are needed for the proofs, as this material is hard to find in the literature. A formula for the equiaffine normal of a level set of $F$ is given in J. Hao and H. Shima's \cite{Hao-Shima} under the condition that the Hessian $F_{ij}$ of $F$ be nondegenerate, which is too restrictive in applications. In particular, in the examples here, $F_{ij}$ has corank one everywhere. Fortunately, the formulas in \cite{Hao-Shima} remain valid under less restrictive hypotheses, provided they are properly interpreted. Let $U^{ij}$ be the adjugate tensor of the Hessian $F_{ij}$ and let $\U(F) = U^{ij}F_{i}F_{j}$, where $F_{i}$ is the differential of $f$. Then the level set of $F$ containing the regular point $p$ is nondegenerate at $p$ if and only if $\U(F)$ is not zero at $p$. This observation is due to R. Reilly in \cite{Reilly-affine}. Since the claim and the paper \cite{Reilly-affine} seem little known, the proof is reviewed here. More generally, there are derived in terms of $U^{ij}$ and $\U(F)$ formulas for the equiaffine normal and equiaffine metric of a level set valid under the hypothesis that $\U(F)$ not vanish. Actually, slightly more is obtained, in that there is constructed from $F$ a vector field $\nm$ that along each level set of $F$ agrees with the equiaffine normal of the level set. The differential operator associating wtih $F$ the vector field $\nm$ is invariant under orientation-preserving external reparameterizations (meaning replacing $F$ by $\Psi\circ F$ for an orientation-preserving diffeomorphism $\Psi$), equivariant (in a particular sense) with respect to the action on functions of the group of affine transformation, and invariant with respect to the action on functions of the group of equiaffine transformations. A further condition is necessary to determine this operator uniquely. See section \ref{prelimsection} for further discussion.

The examples constructed in Section \ref{examplesection} are the level sets of a function $F$ satisfying that its Hessian has corank one and $\U(F)$ is equal to a constant everywhere. The complete description of the solutions of these equations seems an interesting problem on its own. The examples described here show that examples abound, although the equiaffine metrics of the examples have split signature. This should be compared with the situation for convex hypersurfaces with equiaffine mean curvature zero. In \cite{Trudinger-Wang-bernstein}, N. Trudinger and X.~J. Wang have shown the validity of the affine Bernstein conjecture, that a locally uniformly convex hypersurface in $\rea^{3}$ that is complete in the equiaffine metric and has equiaffine mean curvature zero is an elliptic paraboloid. Moreover, in \cite{Trudinger-Wang-survey} they conjecture that the same result should hold for hypersurfaces in $\rea^{n}$ for $n \leq 10$.

That the normal obtained in Section \ref{prelimsection} is the usual equiaffine normal is proved in Theorem \ref{affinenormalsametheorem}. The proof uses a definition of the affine normal distribution of a nondegenerate hypersurface in a manifold with affine connection (not necessarily flat) that was first given in \cite{Fox-ahs}. This derivation is based on requiring a compatibility condition between the connection induced via a transverse line field and the conformal structure determined by the second fundamental form. The precise statement is Theorem \ref{projnormaltheorem}.

One consequence is that it makes sense to speak of the equiaffine normal of a nondegenerate hypersurface in a pseudo-Riemannian space form. Section \ref{comparisonsection} examines briefly the question of the relation between this equiaffine normal and the pseudo-Riemannian unimodular normal. Theorem \ref{normalstheorem} shows that these are proportional if and only if the hypersurface has (nonzero) constant Gauss-Kronecker curvature, and that in this case the equiaffine mean curvature is a constant multiple of the pseudo-Riemannian mean curvature. In particular, a hypersurface of nonzero constant Gauss-Kronecker curvature and constant pseudo-Riemannian mean curvature has constant equiaffine mean curvature. For example, a nondegenerate isoparametric hypersurface has constant equiaffine mean curvature. This raises the question of under what conditions a constant equiaffine mean curvature hypersurface in a space form must be isoparametric.

\section{Equiaffine geometry of level sets}\label{prelimsection}
Let $\hnabla$ be a torsion-free affine connection on the $(n+1)$-dimensional manifold $M$. The \textit{second fundamental form} of a co-orientable immersed hypersurface $\Sigma$ in $M$ with respect to $\hnabla$ is the normal bundle valued symmetric covariant two-tensor on $\Sigma$ equal, when evaluated on vector fields $X$ and $Y$ tangent to $\Sigma$, to the projection of $\hnabla_{X}Y$ onto the normal bundle of $\Sigma$.
The hypersurface $\Sigma$ is \textit{nondegenerate (at a point $p \in \Sigma$)} if its second fundamental form is nondegenerate everywhere (at $p$). A vector field $N$ transverse to $\Sigma$ determines a splitting of the pullback of $TM$ over $\Sigma$ as the direct sum of $T\Sigma$ and the span of $N$. Via this splitting, the connection $\hnabla$ induces on $\Sigma$ a connection $\nabla$, while via $N$, the second fundamental form is represented by a symmetric covariant two tensor $h$ on $\Sigma$. In particular, $\Sigma$ is nondegenerate if and only if $h$ is nondegenerate, and this condition does not depend on the choice of the transversal $N$. For vector fields $X$ and $Y$ tangent to $\Sigma$, the connection $\nabla$, the tensor $h$, the \textit{shape operator} $S \in \Ga(\eno(T\Sigma))$, and the \textit{connection one-form} $\tau \in \Ga(T^{\ast}\Sigma)$ are determined by the relations
\begin{align}\label{induced}
&\hnabla_{X}Y = \nabla_{X}Y + h(X, Y)N,& &\hnabla_{X}N = -S(X) + \tau(X)N,
\end{align}
where here, as in what follows, notation indicating the restriction to $\Sigma$, the immersion, the pullback of $TM$, etc. is omitted to improve readability, and $\Ga(E)$ is the space of smooth sections of the vector bundle $E$. By \eqref{induced}, for $X$ tangent to $\Sigma$, and any volume form $\Psi$ on $M$,
\begin{align}\label{nablavolume}
\nabla_{X}(\imt(N)\Psi)  = \imt(N)\hnabla_{X}\Psi + \tau(X)\imt(N)\Psi.
\end{align}
By \eqref{nablavolume}, if $\hnabla \Psi = 0$, then, along $\Sigma$, $\tau$ is determined by $\tau = (\imt(N)\Psi)^{-1}\nabla(\imt(N)\Psi)$.
The \textit{mean curvature of $\Sigma$ with respect to $N$ and $\hnabla$} is $n^{-1}\tr S$.

\subsection{Nondegeneracy of level sets}\label{ufsection}
Let $M = \rea^{n+1}$ and let $\hnabla$ be the standard flat affine connection on $\rea^{n+1}$.
Fix a $\hnabla$-parallel volume form $\Psi$. The group $\Aff(n+1, \rea)$ of affine transformations of $\rea^{n+1}$ comprises the automorphisms of $\hnabla$. Elements of its subgroup preserving $\Psi^{2}$ are called \textit{unimodular} or \textit{equiaffine}. Let $\Om \subset \rea^{n+1}$ be an open domain. For $F \in C^{\infty}(\Om)$ let $F_{i_{1}\dots i_{k}} = \hnabla_{i_{1}}\dots\hnabla_{i_{k-1}}dF_{i_{k}}$, and let $F_{ij} = (\hess F)_{ij} = \hnabla_{i}dF_{j}$ be the \textit{Hessian} of $F$. Here, as generally in what follows, the abstract index and summation conventions are employed (the reader unfamiliar with these conventions can consult chapter $2$ of \cite{Penrose-Rindler}). As $\det \hess F$ and the tensor square $\Psi^{2}$ are $2$-densities, it makes sense to define the \textit{Hessian determinant} $\H(F)$ of the at least twice differentiable function $F$ by $\det \hess F = \H(F)\Psi^{2}$. By \textit{affine coordinates} are meant smooth functions $x^{1}, \dots, x^{n+1}$ such that the differentials $dx^{1}, \dots, dx^{n+1}$ are linearly independent and constitute a $\hnabla$-parallel coframe; these coordinates are \textit{equiaffine} if, moreover, $\Psi = dx^{1}\wedge \dots \wedge dx^{n+1}$. In equiaffine coordinates, $\H(F) = \det \tfrac{\pr^{2}F}{\pr x^{i}\pr x^{j}}$. 

Formally the adjugate tensor of a symmetric covariant two-tensor $F_{ij}$ is a $2$-density valued symmetric contravariant two-tensor $\bar{U}^{ij}$ satisfying $\bar{U}^{ip}F_{pj} = (\det \hess F) \delta_{j}\,^{i}$. Here, instead, the  symmetric contravariant two-tensor $U^{ij}$ defined by tensoring $\bar{U}^{ij}$ with $\Psi^{-2}$ will be called the \textit{adjugate tensor} of $F_{ij}$. Its characteristic property is $U^{ip}F_{pj} = \H(F)\delta_{j}\,^{i}$. Where $\H(F)$ is nonzero, $F_{ij}$ is a pseudo-Riemannian metric with inverse symmetric bivector $F^{ij}$, and there hold $F^{ij} = \H(F)^{-1}U^{ij}$ and $\H(F)^{-1}U^{ip}F_{p} = F^{ij}F_{j}$. Like $U^{ij}$, the vector field $N^{i} = U^{ip}F_{p}$ and the function $\U(F) = N^{i}F_{i} = U^{ij}F_{i}F_{j}$ are defined even when $F_{ij}$ degenerates. 
The adjugate transformation of an endomorphism of an $r$-dimensional vector space has rank $r$, $1$, or $0$ as the original transformation has rank $r$, $r-1$, or less than $r-1$. Consequently, if $\H(F)$ vanishes then $U^{ij}$ has rank $1$ or $0$ as $F_{ij}$ has rank $n$ or rank less than $n$. When $U^{ij}$ has rank $1$, something more precise can be said; see Lemma \ref{unondegenlemma} below.

Applying the Cauchy determinantal identity (Equation (19) in \cite{Cauchy}),
\begin{align}\label{adjid}
\begin{vmatrix} A & b \\ c^{t} & d\end{vmatrix} = (\det A)d -c^{t}(\adj A)b,
\end{align}
where $A$ is a matrix with adjugate matrix $\adj A$, $b$ and $c$ are column vectors, and $d \in \rea$, yields the identity
\begin{align}\label{rankone}
\begin{split}
\begin{vmatrix} A + bc^{t}\end{vmatrix} & = \begin{vmatrix} A + bc^{t} & b \\ 0 & 1\end{vmatrix} = \begin{vmatrix} A & b \\ -c^{t} & 1\end{vmatrix}
= |A| + c^{t}(\adj A)b.
\end{split}
\end{align}
Applying \eqref{adjid} and \eqref{rankone} to tensors yields 
\begin{align}
\label{ufdet}
&\U(F) = U^{ij}F_{i}F_{j} = -\begin{vmatrix} F_{ij} & F_{i} \\ F_{j} & 0 \end{vmatrix},& \\
\label{fdet}
&\det(F_{ij} + qF_{i}F_{j}) = (\H(F) + q\U(F))\Psi^{2}.
\end{align}
for any smooth function $q$. Notation is abused in \eqref{ufdet} in that a covariant tensor is apparently identified with an endomorphism; if $F_{i}$ and $F_{ij}$ are interpreted as the components of $dF$ and $\hnabla dF$ with respect to equiaffine coordinates the formulas make rigorous sense, and this justifies their use generally. Alternatively the matricial notation can be understood as an abstract notational device like the abstract index conventions. 

\begin{lemma}\label{hgflemma}
For $F \in \cinf(\rea^{n+1})$ and $g \in \Aff(n+1, \rea)$ define $(g \cdot F)(x) = F(g^{-1}x)$. Let $\ell:\Aff(n+1, \rea) \to GL(n+1, \rea)$ be the projection onto the linear part. Then
\begin{align}\label{hgf}
&g\cdot \H(F) = \det{}^{2}\ell(g) \H(g\cdot F),& &g\cdot \U(F) = \det{}^{2} \ell(g) \U(g\cdot F),
\end{align}
\end{lemma}
\begin{proof}
There hold $\ell(g)_{i}\,^{j}(g\cdot F)_{j}(x) = F_{i}(g^{-1}x)$ and $\ell(g)_{i}\,^{a}\ell(g)_{j}\,^{b}(g \cdot F)_{ab}(x) = F_{ij}(g^{-1}x)$. Taking the determinant of the last yields the first equality of \eqref{hgf}, while substituting both into \eqref{ufdet} yields the second identity of \eqref{hgf}.
\end{proof}
\begin{remark}
Let $\Phi:\rea^{n+1} \to \rea^{n+1}$ be a linear fractional transformation, so that $\Phi(x)^{i} = (c_{p}x^{p} + d)^{-1}(A_{q}\,^{i}x^{q} + b^{i})$. Then $\Phi_{j}\,^{i} = \tfrac{\pr}{\pr x^{j}}\Phi^{i}$ and $\Phi_{ij}\,^{k} = \tfrac{\pr^{2}}{\pr x^{i}\pr x^{j}}\Phi^{k}$ satisfy $\Phi_{jk}\,^{i} = -\Phi_{j}\,^{i}c_{k} - \Phi_{k}\,^{i}c_{j}$. By \eqref{ufdet},
\begin{align}\label{ufproj}
\begin{split}
-\U(F\circ \Phi) & = \begin{vmatrix} F_{pq}\Phi_{i}\,^{p}\Phi_{j}\,^{q} + F_{p}\Phi_{ij}\,^{p} & F_{p}\Phi_{i}\,^{p} \\ F_{q}\Phi_{j}\,^{q} & 0 \end{vmatrix} =   \begin{vmatrix} F_{pq}\Phi_{i}\,^{p}\Phi_{j}\,^{q} - 2F_{p}\Phi_{(i}\,^{p}c_{j)} & F_{p}\Phi_{i}\,^{p} \\ F_{q}\Phi_{j}\,^{q} & 0 \end{vmatrix}\\& = \begin{vmatrix} F_{pq}\Phi_{i}\,^{p}\Phi_{j}\,^{q}  & F_{p}\Phi_{i}\,^{p} \\ F_{q}\Phi_{j}\,^{q} & 0 \end{vmatrix} = -(\det T\Phi)^{2}\U(F)\circ \Phi.
\end{split}
\end{align}
The identity \eqref{ufproj} shows that $\U(F)$ transforms equivariantly under precomposition with projective transformations, and yields the second identity of \eqref{hgf} as a special case. This projective covariance of $\U(F)$ is a reason for paying special attention to this quantity.
\end{remark}

A point $p \in \Om$ is a \textit{regular point} of $F$ if $dF$ is not zero at $p$. Because the set of regular points in $\Om$ is open, the level sets of the restriction of $F$ to a sufficiently small neighborhood of a regular point of $F$ are smoothly immersed submanifolds. It makes sense to say that the level set of $F$ containing $p$ is degenerate or nondegenerate at $p$ because the part of this level set contained in a sufficiently small neighborhood of $p$ is smoothly immersed. The \textit{level set of $F$ containing $p$} means $\{x \in \Om: F(x) = F(p)\}$.

A \textit{Euclidean metric} on the flat equiaffine space $(\rea^{n+1}, \Psi, \hnabla)$ means a $\hnabla$-parallel Riemannian metric the volume form of which equals $\Psi$.
Lemma \ref{ufgausslemma} gives a geometric interpretation of $\U(F)$ in terms of the Gauss-Kronecker curvature $\gau$ of a level set of $F$ with respect to a Euclidean metric. 

\begin{lemma}\label{ufgausslemma}
Let $\delta_{ij}$  and $\delta^{ij}$ be a Euclidean metric and its inverse on $\rea^{n+1}$. 
Let $F$ be a $\cinf$ function defined on an open subset $\Om \subset \rea^{n+1}$. At a regular point $p \in \Om$ of $F$, the Gauss-Kronecker curvature $\gau$ with respect to the Euclidean unit normal vector $-|dF|_{\delta}^{-1}\delta^{ip}F_{p}$ of the smoothly immersed level set of $F$ passing through $p$ satisfies
\begin{align}\label{ufgauss}
\U(F) = \gau|dF|^{n+2}_{\delta}.
\end{align}
\end{lemma}

\begin{proof}
Write $E^{i} = -|dF|_{\delta}^{-1}\delta^{ip}F_{p}$ and $E_{i} = E^{p}\delta_{ip} = - |dF|_{\delta}^{-1}F_{i}$. The tensor 
\begin{align}\label{ufpi}
\begin{split}
\Pi_{ij} &=  |dF|^{-1}_{\delta}\left(F_{ij} - 2E^{p}F_{p(i}E_{j)} + E^{p}E^{q}F_{pq}E_{i}E_{j}\right)
\end{split}
\end{align}
satisfies $E^{i}\Pi_{ij} = 0$ and its restriction to the tangent space to a level set of $F$ is the representative of the second fundamental form of the level set with respect to $E^{i}$. The tensor $\La_{ij} = \Pi_{ij} + E_{i}E_{j}$ is nondegenerate and its determinant satisfies $(\det \La)/\Psi^{2} = (\det \Pi)/(\imt(E)\Psi)^{2} = \gau$ where $\det \Pi$ means the determinant of the restriction of $\Pi$ to a tangent space of a level set of $F$, and $\gau$ is the Gauss-Kronecker curvature with respect to $E$ ($\gau$ is defined by the preceding relation). Elementary operations with determinants coupled with \eqref{ufpi} and \eqref{ufdet} yield (abusing notation as in \eqref{ufdet})
\begin{align}\label{ufg1}
\begin{split}
-\gau & = \begin{vmatrix} \La_{ij} & E_{j}\\ 0 & -1 \end{vmatrix} =  \begin{vmatrix} \Pi_{ij} & E_{j}\\ E_{i} & -1 \end{vmatrix}  =  \begin{vmatrix} \Pi_{ij} & E_{j}\\ E_{i} & 0 \end{vmatrix} = \begin{vmatrix} |dF|_{\delta}^{-1}F_{ij} & E_{j}\\ E_{i} & 0 \end{vmatrix} = -|dF|_{\delta}^{-n-2}\U(F),
\end{split}
\end{align}
from which \eqref{ufgauss} follows. 
\end{proof}

\begin{corollary}[Corollary of Lemma \ref{nondegenlemma}]\label{gkcorollary}
Let $F$ be a $\cinf$ function defined on an open subset $\Om \subset \rea^{n+1}$ and satisfying $\U(F) = 0$ on $\Om$. If $r \in \rea$ is a regular value of $F$ and the level set $\lc_{r}(F, \Om) = \{x \in \Om: F(x) = r\}$ is nonempty, then $\lc_{r}(F, \Om)$ is a smoothly immersed hypersurface of Gauss-Kronecker curvature zero.
\end{corollary}
\begin{proof}
Since $r$ is regular, $|dF|_{\delta}$ does not vanish on $\lc_{r}(F, \Om)$ and the claim follows from \eqref{ufgauss}.
\end{proof}

\begin{lemma}\label{nondegenlemma}
Let $F$ be a $\cinf$ smooth function on a nonempty open subset $\Om \subset \rea^{n+1}$. 
For $p \in \Omega$ the following are equivalent.
\begin{enumerate}
\item\label{nd1} $\U(F)$ is not zero at $p$.
\item\label{nd2} $p$ is a regular point of $F$ and the Gauss-Kronecker curvature of $F$ at $p$ with respect to any Euclidean metric is nonzero.
\item\label{nd3} $p$ is a regular point of $F$ and the level set of $F$ containing $p$ is nondegenerate at $p$.
\item\label{nd4} $p$ is a regular point of $F$ and the restriction of the Hessian of $F$ to the level set of $F$ containing $p$ is nondegenerate at $p$. 
\end{enumerate}
\end{lemma}
\begin{proof}
The equivalence of \eqref{nd1} and \eqref{nd2} is immediate from \eqref{ufgauss}. The equivalence of \eqref{nd2} and \eqref{nd3} results from the observations that the shape operator of a Riemannian metric is invertible if and only if the representative of the second fundamental form associated with a unit normal vector field is nondegenerate and that the second fundamental form is nondegenerate if and only if its representative with respect to any transversal is nondegenerate. There remains to show the equivalence of \eqref{nd3} and \eqref{nd4}. If $p$ is a regular point then the level set $\Sigma$ containing $p$ of the restriction of $F$ to a small neighborhood of $p$ is smoothly immersed and there is a vector field $V$ transverse to $\Sigma$ near $p$ and satisfying $F_{i}V^{i} \neq 0$ in a neighborhood in $\Sigma$ of $p$. Let $h$ be the representative of the second fundamental form of $\Sigma$ associated with $V$ and let $X$ and $Y$ be vector fields tangent to $\Sigma$ near $p$. Then, along $\Sigma$, $(\hnabla_{X}dF)(Y) = -dF(\hnabla_{X}Y) = -h(X,Y)dF(V)$. Since $dF(V) \neq 0$ at $p$, it follows that $h$ is nondegenerate if and only if the restriction of $F_{ij}$ to $\Sigma$ is nondegenerate at $p$.
\end{proof}
\begin{remark}
The main point of Lemma \ref{nondegenlemma} is the equivalence of \eqref{nd1} and \eqref{nd2}. This was proved by R. Reilly as Proposition $4$ of \cite{Reilly-affine}.
\end{remark}

\begin{remark}
The deduction of the equivalence of \eqref{nd1} and \eqref{nd3} of Lemma \ref{nondegenlemma} via the identity \eqref{ufgauss} relating $\U(F)$ to the Gauss-Kronecker curvature is not completely satisfying because of its use of an apparently extraneous Euclidean structure. Here is given a direct proof that \eqref{nd3} implies \eqref{nd1} using only the equivalence of \eqref{nd3} and \eqref{nd4} and making no use of any metric structure. (An alternative proof that \eqref{nd1} implies \eqref{nd3} is given in the proof of Lemma \ref{unondegenlemma} below.) It is claimed that if $\U(F)$ vanishes at a regular point $p$ then there is a vector field $X$ tangent to $\Sigma$ at $p$ and such that $X^{p}F_{ip} = 0$. By the equivalence of \eqref{nd3} and \eqref{nd4} of Lemma \ref{nondegenlemma}, this suffices to show that $\Sigma$ is degenerate at $p$. If $\U(F)$ vanishes at $p$ and $N^{i}$ does not vanish at $p$, then for any vector field $Y$ tangent to $\Sigma$ at $p$ there holds $0 = \H(F)dF(Y) = (\hnabla_{Y}dF)(N)$, which shows that $F_{ij}$ degenerates at $p$. If $U^{ij}$ vanishes at $p$ then the rank of $F_{ij}$ at $p$ is at most $n-1$, so its restriction to $\Sigma$ is degenerate at $p$. Finally, if $N$ vanishes at $p$ and $U^{ij}$ has rank $1$ at $p$ then there are a vector field $X$ and a smooth function $c$ defined in a neighborhood of $p$ such that at $p$ there holds $U^{ij} = cX^{i}X^{j}$. If $p$ is a regular point, then, since $0 = N^{i} = U^{ij}F_{j} = cX^{i}X^{p}F_{p}$, $X$ is tangent to $\Sigma$ at $p$. On the other hand, since $0 = U^{ip}F_{pj} = cX^{i}X^{p}F_{ip}$ at $p$, there holds $X^{p}F_{ip} = 0$ at $p$. This shows that the negation of \eqref{nd1} implies the negation of \eqref{nd3}; precisely, if $\U(F)$ vanishes at $p$ then $\Sigma$ is degenerate at $p$ or $p$ is a critical point, and the latter possibility can occur only if $N$ vanishes at $p$.
\end{remark}

\begin{lemma}\label{unondegenlemma}
Let $F$ be a $\cinf$ function defined on an open subset of $\rea^{n+1}$ and let $\Om$ be a connected component with nonempty interior of the set where $\U(F)$ does not vanish. Define $\sff_{ij} =  F_{ij} - \U(F)^{-1}\H(F)F_{i}F_{j}$. Suppose $\lc_{r}(F, \Omega) = \{x \in \Omega: F(x) = r\}$ is nonempty. Then:
\begin{enumerate}
\item\label{msff} $N^{i} = U^{ip}F_{p}$ does not vanish on $\Om$ and each nonempty level set $\lc_{r}(F, \Om)$ is a smoothly immersed nondegenerate hypersurface co-oriented by $N$. The representative of the second fundamental form of $\lc_{r}(F, \Om)$ with respect to the transversal $N$ is the restriction of
\begin{align}\label{hn}
h_{ij} = -\U(F)^{-1}\sff_{ij} = -\U(F)^{-1}\left(F_{ij} - \H(F)\U(F)^{-1}F_{i}F_{j}\right),
\end{align}
and the signature of $h_{ij}$ is constant on $\Om$. 
\item\label{upij} For any nonvanishing $q \in \cinf(\Om)$ the tensor
\begin{align}
\label{kdef}m_{ij} &= \sff_{ij} + q\U(F)^{-1}F_{i}F_{j} = F_{ij} - \U(F)^{-1}\H(F)F_{i}F_{j} + q\U(F)^{-1}F_{i}F_{j},
\end{align}
is nondegenerate and $\det m = q \Psi^{2}$. Let $m^{ij}$ be the inverse of $m_{ij}$ defined by $m^{ip}m_{pj} = \delta_{j}\,^{i}$. The tensor $\sff^{ij} =  m^{ij} - q^{-1}\U(F)^{-1}N^{i}N^{j}$ does not depend on the choice of $q$. Since $\sff^{ip}F_{p} = 0$, it makes sense to speak of the restriction of $\sff^{ij}$ to $\lc_{r}(F, \Om)$, and this restriction is the inverse to the restriction of $\sff_{ij}$ to $\lc_{r}(F, \Om)$.
\item\label{upij0} There holds $U^{ij} - \U(F)^{-1}N^{i}N^{j} = \H(F)\sff^{ij}$. 
At a point $p \in \Om$ where $\H(F)$ vanishes, $U^{ij} = \U(F)^{-1}N^{i}N^{j}$ and $F_{ij}$ has rank $n$.
\item If $\Delta = \{x \in \Om:\H(F)(x) \neq 0\}$ has nonempty interior, then $\H(F)^{-1}(U^{ij} - \U(F)^{-1}N^{i}N^{j})$ extends smoothly to the closure $\bar{\Delta}$, where it equals $\sff^{ij}$.
\item\label{uvolrelation} Along $\lc_{r}(F, \Om)$, there holds $|\U(F)|^{(n+1)/2}|\vol_{h}| = |\imt(N)\Psi|$ where $|\vol_{h}|$ is the volume density induced on $\lc_{r}(F, \Om)$ by the metric $h_{ij}$ of \eqref{hn}.
\end{enumerate}
\end{lemma}
\begin{proof}
Since $dF$ does not vanish on $\Om$, if $\lc_{r}(F, \Om)$ is nonempty then it is a smoothly immersed hypersurface, and since $N^{i}F_{i} = \U(F)$ also does not vanish on $\Om$, $N^{i}$ is transverse to $\lc_{r}(F, \Om)$. By \eqref{induced}, for vector fields $X$ and $Y$ tangent to $\lc_{r}(F, \Om)$ there holds $(\hnabla_{X}dF)(Y) = -\U(F)h(X, Y)$. Together with $N^{p}F_{ip} = \H(F)F_{i}$ this shows that the tensor \eqref{hn} satisfies $N^{p}h_{ip} = 0$ and that its restriction to $\lc_{r}(F, \Om)$ is the representative of the second fundamental form determined by $N^{i}$. 
Applying \eqref{fdet} yields $\det m = q\Psi^{2}$, so $m_{ij}$ is nondegenerate. Because $\Om$ is connected and, in the space of symmetric bilinear forms on a vector space, a connected component of the subspace of nondegenerate forms comprises forms of a fixed signature, the smooth nondegenerate form $m_{ij}$ cannot change signature on $\Om$. Since $m_{ij} = -\U(F)h_{ij} + \U(F)^{-1}F_{i}F_{j}$, $N^{p}m_{ip} = qF_{i}$ annihilates the tangent space to $\lc_{r}(F, \Om)$, and $N^{i}N^{j}m_{ij} = \U(F)q$ has constant sign on $\Om$, the tensor $h_{ij}$ cannot change signature on $\Om$.

Let $\tilde{m}_{ij}$ be defined as $m_{ij}$ in \eqref{kdef}, but with the nonvanishing function $\tilde{q} \in \cinf(\Om)$ in place of $q$. Then 
\begin{align}
\begin{split}
\tilde{m}^{ij} - m^{ij} & = 
(\tilde{m}^{ip} - m^{ip})m_{pq}m^{qj} \\
&= \tilde{m}^{ip}(\tilde{m}_{pq} + (q - \tilde{q})\U(F)^{-1}F_{p}F_{q})m^{qj} - m^{ij} = (\tilde{q}^{-1} - q^{-1})\U(F)^{-1}N^{i}N^{j},
\end{split}
\end{align} 
from which it follows that the tensor $\sff^{ij}$ does not depend on $q$.

The tensor $\sff_{ij}$ satisfies $N^{j}\Pi_{ij} = 0$, so $N^{j}m_{ij} = qF_{i}$. Consequently, $N^{i}= qm^{ip}F_{p}$, and so $\sff^{ij}F_{j} = 0$. Hence it makes sense to speak of the restriction of $\sff^{ij}$ to $\lc_{r}(F, \Om)$ . Since $\sff^{ij}F_{j} = 0$,
\begin{align}\label{sffsff}
\sff^{ip}\sff_{pj} = \sff^{ip}m_{pj} = (m^{ip} - q^{-1}\U(F)^{-1}N^{i}N^{p})m_{pj} = \delta_{j}\,^{i} - \U(F)^{-1}F_{j}N^{i},
\end{align}
which shows that the restriction of $\sff^{ij}$ to $\lc_{r}(F, \Om)$ is the symmetric tensor inverse to the restriction to $\lc_{r}(F, \Om)$ of $\sff_{ij}$. Because $(U^{ip} - \U(F)^{-1}N^{i}N^{p})F_{p} = 0$,
\begin{align}\label{unondegen0}
(U^{ip} - \U(F)^{-1}N^{i}N^{p})m_{pj} =(U^{ip} - \U(F)^{-1}N^{i}N^{p})F_{pj} = \H(F)(\delta_{j}\,^{i} - \U(F)^{-1}F_{j}N^{i}) 
\end{align}
Raising the index $j$ in \eqref{unondegen0} and substituting $N^{i} = qm^{ip}F_{p}$ into the result yields \eqref{upij0}.
By \eqref{upij0}, where $\H(F)$ vanishes there holds $U^{ij} = \U(F)^{-1}N^{i}N^{j}$. 
Since $\sff^{ij}$ is smooth on $\Om$, it follows from \eqref{upij0} that, whenever the subset $\Delta$ is nonempty, the tensor $\H(F)^{-1}(U^{ij} - \U(F)^{-1}N^{i}N^{j})$ extends smoothly to the closure $\bar{\Delta}$, where it equals $\sff^{ij}$. 

Let $m_{ij}$ be the tensor defined in \eqref{kdef} with $q = \U(F)^{-1}$, so that, by \eqref{upij}, $\det m = \Psi^{2}$ and $m^{ij}F_{i}F_{j} = \U(F)$. By the definition of the volume densities $|\vol_{h}|$ and $|\vol_{m}|$ of the metrics $h$ and $m$, for any $X_{1}, \dots, X_{n}$ tangent to $\Sigma$ and $u = |\U(F)|^{-1/2}$ there holds
\begin{align}\label{quf}
\begin{split}
|(\imt(\gn)&\Psi)(X_{1}, \dots, X_{n})|  = u^{-(n+1)}\left|\Psi\left(u\gn, uX_{1}, \dots, uX_{n}\right)\right| 
 \\ &= u^{-(n+1)}\left|\vol_{m}\left(u\gn, uX_{1}, \dots, uX_{n}\right)\right|  =  |\U(F)|^{(n+1)/2}|\vol_{h}(X_{1}, \dots, X_{n})|.
\end{split}
\end{align}
This shows \eqref{uvolrelation}.
\end{proof}

\begin{remark}
A consequence of Lemma \ref{unondegenlemma} is that when $\H(F) = 0$ on all of $\Om$ the tensor $\sff^{ij}$ is still defined. This means that formulas obtained assuming $\H(F) \neq 0$ and involving $\H(F)^{-1}(U^{ij}  - \U(F)^{-1}N^{i}N^{j})$ continue to make sense where $\H(F)$ vanishes, provided that $\U(F)$ does not vanish.
\end{remark}

\begin{example}
Allowing $\H(F)$ to vanish is useful because in interesting examples it occurs that the level sets of $F$ are nondegenerate although $\H(F)$ is identically zero. For example, along the helicoid defined by the vanishing of $F(u, x, y) = x\sin u + y\cos u$, the Hessian of $F$ degenerates, but it follows from Lemma \ref{unondegenlemma} that this level set is a nondegenerate hypersurface, because $\U(F) = -1$. By Corollary \ref{amccorollary} below, the helicoid has equiaffine mean curvature zero. In section \ref{examplesection} this example is generalized to higher dimensions.
\end{example}

Example \ref{gnexample} shows that it can occur that at a regular point $p$ of $F$ the Hessian of $F$ has corank one and $\U(F)$ vanishes. Lemma \ref{nondegenlemma} implies that in this case the level set of $F$ containing $p$ is degenerate at $p$.

\begin{example}\label{gnexample}
It can occur that $U^{ij}$ has rank $1$ (so $\hess F$ has corank $1$ and $\H(F) = 0$) and $\U(F)$ vanishes. The following example is based on the construction in section $7$ of \cite{Gordan-Noether}. 
Let $k \geq 2$ and let $a$, $b$, and $c$ be linearly independent homogeneous degree $k$ polynomials of the variables $x_{1}$ and $x_{2}$. Then $P = a(x_{1}, x_{2})x_{3} + b(x_{1}, x_{2})x_{4} + c(x_{1}, x_{2})x_{5} \in \pol^{k+2}(\rea^{5})$ is irreducible and not affinely equivalent to a polynomial of less than five variables but solves $\H(P) = 0$, for the kernel of the Hessian of $P$ contains the vector field 
\begin{align}\label{vgn}
V = (b_{1}c_{2} - b_{2}c_{1})\pr_{3} + (c_{1}a_{2} - c_{2}a_{1})\pr_{4} + (a_{1}b_{2} - a_{2}b_{1})\pr_{5}, 
\end{align}
where subscripts indicate first partial derivatives. A concrete example is $P = x_{1}^{2}x_{3} + x_{1}x_{2}x_{4} + x_{2}^{2}x_{5}$. See \cite{Ciliberto-Russo-Simis} for many related examples. Straightforward calculations show that $U^{ij} = V^{i}V^{j}$, so that $\U(P) = dP(V)^{2}$. However, 
\begin{align}\label{dpv}
dP(V) = \begin{vmatrix} a & a_{1} & b_{1} \\ b & b_{1} & b_{2} \\ c & c_{1} & c_{2} \end{vmatrix} = 0.
\end{align}
is the determinant of the $1$-jet of the mapping from $\rea^{2}$ to $\rea^{3}$ having components $a$, $b$, and $c$, and this vanishes because, by the Euler identity, $ka = x_{1}a_{1} + x_{2}a_{2}$ and similarly for $b$ and $c$.
Hence $\U(P) = 0$ although $U^{ij}$ has rank $1$. 

Since the Euclidean norm $|dP|_{\delta}$ does not vanish along a regular level set $\Sigma$ of $P$, by Corollary \ref{gkcorollary} the Gauss-Kronecker curvature of such a level set is zero. That is $\Sigma$ is a developable hypersurface.
\end{example}

Example \ref{detexample}, at the end of Section \ref{equiaffinenormalsection}, illustrates some issues related to connected components of level sets and the signatures of their second fundamental forms.

\subsection{Equiaffine normal vector field}\label{equiaffinenormalsection}
This section describes the equiaffine normal, and its associated tensors, of a level set of a function $F \in \cinf(\Om)$, where $\Om \subset \rea^{n+1}$ is an open set, satisfying that $\U(F)$ does not vanish on $\Om$.

\begin{lemma}\label{affinenormaldefinitionlemma} 
Let $F$ be a $\cinf$ function defined on an open subset of $\rea^{n+1}$. Let $\Omega$ be a connected component with nonempty interior of the region on which $\U(F)$ does not vanish. Then a nonempty level set $\lc_{r}(F, \Om)$ is nondegenerate and the vector field $\nm$ defined by
\begin{align}\label{nmdefined}
\begin{split}
\nm^{i} &= -|\U(F)|^{1/(n+2)}\U(F)^{-1}N^{i} - k^{ip}\muf_{p},
\end{split}
\end{align}
where  $k^{ij} = |\U(F)|^{1/(n+2)}\sff^{ij}$, $\sff^{ij}$ is defined in \eqref{upij} of Lemma \ref{unondegenlemma}, and $\muf = (n+2)^{-1}d\log\U(F)$, is transverse to $\lc_{r}(F, \Om)$. The restriction to $\lc_{r}(F, \Om)$ of the tensor 
\begin{align}\label{emnm}
k_{ij} = |\U(F)|^{-1/(n+2)}\sff_{ij} = |\U(F)|^{-1/(n+2)}(F_{ij} - \U(F)^{-1}\H(F)F_{i}F_{j}),
\end{align}
represents the second fundamental form of $\lc_{r}(F, \Om)$ with respect to $\nm$. 
\end{lemma}
\begin{proof}
The nondegeneracy of the level sets of $F$ is the conclusion of Lemma \ref{unondegenlemma}. By \eqref{msff} of Lemma \ref{unondegenlemma}, $k_{ij}$ represents the second fundamental form of $\lc_{r}(F, \Om)$ with respect to $\nm$. 
\end{proof}

The transversal $\nm^{i}$ is the \textit{equiaffine normal vector field} associated with $F$. This terminology is justified by Theorem \ref{affinenormalsametheorem} that shows that, along a level set of $F$, $\nm^{i}$ agrees with the equiaffine normal of the level set as usually defined. By Lemma \ref{unondegenlemma}, the tensor $k_{ij}$ of \eqref{emnm} has rank $n$ wherever $\U(F)$ is not zero, and, since $N^{j}k_{ij} = 0$, by \eqref{hn} and \eqref{nm1}, the restriction of $k_{ij}$ to a level set of $F$ is the representative of the second fundamental form determined by $\nm$; it is the \textit{equiaffine metric} of the level set. Note that, although the signature of $k_{ij}$ need not be the same on distinct connected components of a level set of $F$, on those contained within a connected component of the complement of the zero set of $\U(F)$, it does not change, by \eqref{msff} of Lemma \ref{unondegenlemma}.

\begin{remark}
After translating notation, the formula \eqref{nm1} for the equiaffine normal can be identified, up to sign, with that obtained, under the assumption $\H(F) \neq 0$, by J. Hao and H. Shima in Theorem $1$ of \cite{Hao-Shima}. Precisely, in the notation used here the formula of \cite{Hao-Shima} can be written 
\begin{align}\label{hsnm}
\begin{split}
 \sign(\U\H)|\U|^{1/(n+2)}\left(\U^{-1}\H F^{ip}F_{p} + F^{ip}\muf_{p} - \U^{-1}\H F^{pq}F_{p}\muf_{q}F^{ia}F_{a}\right)
\end{split}
\end{align}
where $\H = \H(F)$, $\U = \U(F)$, and $F^{ij} = \H(F)^{-1}U^{ij}$ is the symmetric tensor satisfying $F^{ip}F_{pj} = \delta_{j}\,^{i}$ (which exists because $\H(F) \neq 0$). It is not obvious that the expression \eqref{hsnm} continues to make sense when $\H(F) = 0$ but $\U(F) \neq 0$. That this is so follows from Lemma \ref{unondegenlemma}; precisely, \eqref{hsnm} equals $-\sign(\U(F)\H(F))\nm^{i}$. 
\end{remark}

Let $I\subset \rea$ be a connected open interval and let $\psi:I \to \rea$ be a $C^{2}$ diffeomorphism. The level sets of $F$ and $\psi \circ F$ are the same, just differently parameterized, in the sense that for $r \in I$ there holds $\lc_{r}(F,  F^{-1}(I)\cap \Omega) = \lc_{\psi(r)}(\psi \circ F, \Omega)$, and so objects depending only on the geometry of the level sets need to transform well under such external reparameterizations. 

\begin{lemma}\label{uinflemma}
For an open domain $\Om \subset \rea^{n+1}$, define $\uinf(\Om) = \{F \in \cinf(\Om): \U(F)(x) \neq 0 \,\,\text{for all}\,\, x \in \Om\}$. If $F \in \uinf(\Om)$, $I \subset F(\Om)$ is a connected open subinterval, and $\psi:I \to \rea$ is a smooth diffeomorphism onto its image then $\psi \circ F \in \uinf(\Om)$.
\end{lemma}
\begin{proof}
There holds
\begin{align}\label{gpsif}
(\psi \circ F)_{ij} = \dot{\psi}F_{ij} + \ddot{\psi}F_{i}F_{j}. 
\end{align}
From \eqref{ufdet} there results
\begin{align}\label{upsif}
\begin{split}
\U(\psi \circ F) &= -\begin{vmatrix} (\psi \circ F)_{ij} & (\psi \circ F)_{i} \\ (\psi \circ F)_{j} & 0 \end{vmatrix}  = -\begin{vmatrix} \dot{\psi}F_{ij} + \ddot{\psi}F_{i}F_{j}  & \dot{\psi}F_{i} \\ \dot{\psi}F_{j} & 0 \end{vmatrix} \\
& =  -\dot{\psi}^{n+2}\begin{vmatrix} F_{ij}  & F_{i} \\ F_{j} & 0 \end{vmatrix} = \dot{\psi}^{n+2}\U(F).
\end{split}
\end{align}
By \eqref{upsif}, the assumption that $\U(F)$ not vanish on $\Om$ is preserved by external reparameterization.
\end{proof}

\begin{remark}
Computing the determinant of \eqref{gpsif} using \eqref{rankone} yields
\begin{align}\label{hpsif}
\H(\psi \circ F) = \dot{\psi}^{n+1}(\H(F) + (\ddot{\psi}/ \dot{\psi})\U(F)).
\end{align}
By Lemma \ref{uinflemma}, $\uinf$ is preserved by external reparameterizations whereas, by \eqref{hpsif}, the analogous statement is not true for the subset comprising $F \in \cinf(\Om)$ with nonvanishing $\H(F)$. This observation gives another reason for the prominent role played by $\U(F)$.
\end{remark}

An operator $\A:\uinf(\Om) \to \Gamma(T\Om)$ is \textit{invariant under (smooth) external reparameterizations} if for every $F \in \uinf(\Om)$, every connected open subset $I \subset F(\Om)$, and every smooth $\psi:I \to \rea$ mapping $I$ diffeomorphically onto its image there holds
\begin{align}\label{extre}
\A(\psi \circ F)^{i} = \sign(\dot{\psi})\A(F)^{i}.
\end{align}
Let $L_{g}$ be the operator of left multiplication by $g \in \Aff(n+1, \rea)$ on $\rea^{n+1}$, so that $g \cdot F = F \circ L_{g^{-1}}$. Then $\A$ is \textit{equiaffinely invariant} if for every $g \in \Aff(n+1, \rea)$ there holds
\begin{align}\label{affcov}
L_{g^{-1}}^{\ast}(\A(F))^{i} = |\det \ell(g)|^{2/(n+2)}\A(g\cdot F)^{i}
\end{align}
at every point $x \in \Om$ such that $g^{-1}x \in \Om$. Note that \eqref{affcov} actually demands that $\A$ be affinely covariant in a particular way. There could be considered a rule like \eqref{affcov} with an arbitrary character $\chi:\Aff(n+1, \rea) \to \reat$ in place of $|\det \ell(g)|^{2/(n+2)}$. The particular choice of character is explained as follows. The interest here is in operators $\A(F)$ that are \textit{transverse} in the sense that they satisfy additionally the condition that $dF_{i}\A(F)^{i}$ is nonvanishing on $\Om$ for $F \in \uinf(\Om)$. This guarantees that at each $p \in \Om$ the vector field $\A(F)^{i}$ is transverse to the level set $\Sigma$ of $F$ containing $p$. On $\Sigma$ there are two natural volume densities determined by $\A(F)$. These are $|\imt(\A(F))\Psi|$ and the volume density of the representative of the second fundamental form of $\Sigma$ corresponding to $\A(F)$. The choice of character in \eqref{affcov} is determined by requiring that these two volume densities rescale in the same way when $F$ is replaced by $g \cdot F$. Precisely, suppose $L_{g^{-1}}^{\ast}(\A(F)) = \chi(g)\A(g\cdot F)$. Then $|\imt(\A(g\cdot f)\Psi| = |\det \ell(g)|^{-1}|\chi(g)||\imt(\A(F))\Psi|$ and the volume densities $|\vol_{\tilde{h}}|$ and $|\vol_{h}|$ of the representatives $\tilde{h}$ and $h$ of the second fundamental form of $\Sigma$ corresponding respectively to $\A(g\cdot F)$ and $\A(F)$ are related by $|\vol_{\tilde{h}}| = |\chi(g)|^{-n/2}|\vol_{h}|$. Hence $|\imt(\A(g\cdot f))\Psi|/|\vol_{\tilde{h}}| = |\imt(\A(F)\Psi|/|\vol_{h}|$ if and only if $|\chi(g)| = |\det \ell(g)|^{2/(n+2)}$.

\begin{theorem}\label{affinecovariancetheorem}
For an open domain $\Om \subset \rea^{n+1}$, the transverse operators associating with $F \in \uinf(\Om)$ the vector fields $\nm^{i}$ (where $\nm$ is defined in \eqref{nmdefined}) and
$-|\U(F)|^{1/(n+2)}\U(F)^{-1}N^{i}$ on $\Om$ are invariant under external reparameterizations and equiaffinely invariant. 
\end{theorem}

\begin{proof}
For readability there are written $\dot{\psi}$, $\ddot{\psi}$, etc. instead of $\dot{\psi}(F)$, $\ddot{\psi}(F)$, etc. 
The objects derived from $\psi \circ F$ in the same manner as those derived from $F$ are decorated with a $\tilde{\,\,\,}$. For example \eqref{gpsif} is written $\tilde{F}_{ij} = \dot{\psi}F_{ij} + \ddot{\psi}F_{i}F_{j}$. By \eqref{gpsif} and \eqref{hpsif}, the tensor $\sff_{ij} = F_{ij} - \H(F)\U(F)^{-1}F_{i}F_{j}$ defined in Lemma \ref{unondegenlemma} transforms by
\begin{align}\label{hnpsif}
\tilde{\sff}_{ij} = \dot{\psi}\sff_{ij}.
\end{align}
By Lemma \ref{unondegenlemma}, the restriction of $\tilde{\sff}^{ij}$ to a level set of $(\psi \circ F)$ is the inverse of the restriction of $\tilde{\sff}_{ij} = \dot{\psi}\sff_{ij}$ to this level set. Since this level set is also a level set of $F$ and the restriction to it of $\tilde{\sff}^{ij}$ is the inverse of the restriction of $\sff_{ij}$ it follows that 
\begin{align}\label{uppsif}
\tilde{\sff}^{ij} = \dot{\psi}^{-1}\sff^{ij}.
\end{align}
By \eqref{upsif} and \eqref{hnpsif}, the tensor $k_{ij} = |\U(F)|^{-1/(n+2)}\sff_{ij}$ defined in \eqref{emnm}
transforms as $\tilde{k}_{ij} = \sign(\dot{\psi})k_{ij}$. Since the tensor $k^{ij} = |\U(F)|^{1/(n+2)}\sff^{ij}$ satisfies $k^{ip}F_{p} = 0$ it makes sense to speak of the restriction of $k^{ij}$ to a level set of $F$, and, since $k^{ip}k_{pj} = \delta_{j}\,^{i} - \U(F)^{-1}F_{j}N^{i}$, this restriction is the inverse of the restriction of $k_{ij}$.
By \eqref{sffsff}, \eqref{upsif}, \eqref{hnpsif}, and \eqref{uppsif}
\begin{align}
\begin{split}
\U(F)^{-1}F_{i}N^{j} &= \delta_{i}\,^{j} - \sff^{jp}\sff_{ip} = \delta_{i}\,^{j} - \tilde{\sff}^{jp}\tilde{\sff}_{ip} \\ &= \U((\psi \circ F))^{-1}\tilde{F}_{i}\tilde{N}^{j} = \dot{\psi}^{-n-1}\U(F)^{-1}F_{i}\tilde{N}^{j}.
\end{split}
\end{align}
Since $\U(F)^{-1}F_{i}$ does not vanish, this shows
\begin{align}\label{psifn}
\tilde{N}^{i} = \tilde{U}^{ip}(\psi \circ F)_{p} = \dot{\psi}^{n+1}N^{i}.
\end{align}
Together \eqref{upsif} and \eqref{psifn} show that $-|\U(F)|^{1/(n+2)}\U(F)^{-1}N^{i}$ satisfies \eqref{extre}.
It follows from \eqref{upsif} that the one-form $\muf = (n+2)^{-1}d\log\U(F)$ transforms by 
\begin{align}\label{mupsif}
\muf(\psi \circ F)_{i} = \muf_{i} + (\ddot{\psi}/ \dot{\psi})F_{i}.
\end{align}
The vector field $Z^{i} = \U(F)\sff^{ip}\muf_{p}$ satisfies $Z^{p}F_{p} = 0$ and $Z^{p}\sff_{ip} = \U(F)\muf_{i} - N^{p}\muf_{p}F_{i}$. By \eqref{upsif}, \eqref{uppsif}, and \eqref{mupsif}, $\tilde{Z}^{i} = \dot{\psi}^{n+1}Z^{i}$. By \eqref{hgf}, \eqref{hpsif}, \eqref{upsif}, and \eqref{psifn}, the vector field
\begin{align}\label{nm1}
\begin{split}
\nm^{i} &= - \sign(\U(F))|\U(F)|^{-(n+1)/(n+2)}\left(N^{i} + Z^{i}\right) = -|\U(F)|^{1/(n+2)}\U(F)^{-1}N^{i} - k^{ip}\muf_{p},
\end{split}
\end{align}
transforms by $\tilde{\nm}^{i} = \sign(\dot{\psi})\nm^{i}$. This shows that $\nm$ is invariant under external reparameterizations.

There remains to check \eqref{affcov}. Now, write $\tilde{F}_{ij} = (g \cdot F)_{ij}$ for $g \in \Aff(n+1, \rea)$. More generally, a tensor or scalar constructed from $g\cdot F$ is indicated using the same notation as that indicating the corresponding tensor or scalar constructed from $F$, but decorated with a $\tilde{\,}\,$. By construction $L_{g}^{\ast}(d(g\cdot F)) = dF$ and $L_{g}^{\ast}(\hess(g\cdot F)) = \hess F$, where $L_{g}$ indicates the action by left multiplication of $g \in \Aff(n+1, \rea)$ on $\rea^{n+1}$. It is convenient to write $\ell_{i}\,^{j} = \ell(g)_{i}\,^{j}$ and $\bar{\ell}_{i}\,^{j}$ for the inverse endomorphism (so $\bar{\ell}_{p}\,^{j}\ell_{i}\,^{p} = \delta_{i}\,^{j}$). Then $(g\cdot F)_{i}(x) = \bar{\ell}_{i}\,^{p}F_{p}(g^{-1}x)$ and $\tilde{F}(x)_{ij} = \bar{\ell}_{i}\,^{p}\bar{\ell}_{j}\,^{q}F_{pq}(g^{-1}x)$. With \eqref{hgf} it follows that $\tilde{U}^{ij}(x) = (\det \ell(g))^{-2}\ell_{a}\,^{i}\ell_{b}\,^{j}U^{ab}(g^{-1}x)$ and so $\tilde{N}^{i}(x) = (\det \ell(g))^{-2}\ell_{a}\,^{i}N^{a}(g^{-1}x)$. Equivalently, $(\det \ell(g))^{-2}L_{g^{-1}}^{\ast}(N)^{i} = \tilde{N}^{i}$. This suffices to show that $-|\U(F)|^{1/(n+2)}\U(F)^{-1}N^{i}$ satisfies \eqref{affcov}. From \eqref{hgf} it follows that the one-form $(n+2)\muf(F)_{i} = d_{i}\log \U(F)$ satisfies $\muf(g\cdot F) = L_{g^{-1}}^{\ast}\muf(F)$. Similarly, $\tilde{\sff}_{ij}(x) = \bar{\ell}_{i}\,^{p}\bar{\ell}_{j}\,^{q}\sff_{pq}(g^{-1}x)$. From \eqref{sffsff} and \eqref{hgf} it follows that $\tilde{\sff}^{ij} = \ell_{p}\,^{i}\ell_{q}\,^{j}\sff^{pq}(g^{-1}x)$. It follows that $\tilde{Z}^{i}(x) = (\det \ell(g))^{-2}\ell_{a}\,^{i}Z^{a}(g^{-1}x)$. Assembling the preceding shows $\nm$ satisfies \eqref{affcov}.

Alternatively, Theorem \ref{affinenormalsametheorem} below shows that along a level set of $F$ the transversal $\nm^{i}$ agrees with the equiaffine normal of the level set, as usually defined, and the affine covariance \eqref{affcov} is true by construction for the usual affine normal, so there is really no need to check it directly. (Note that the same argument does not apply to show the equiaffine invariance of $-|\U(F)|^{1/(n+2)}\U(F)^{-1}N^{i}$.)
\end{proof}

\begin{example}\label{graphsection}
Here it is shown how to recover from \eqref{nm1} the usual formula (see, for example, equation $(3.4)$ in \cite{Nomizu-Sasaki}) for the equiaffine normal of a graph. 

Let $\ste$ be an $(n+1)$-dimensional vector space equipped with the standard equiaffine structure $(\nabla, \Psi)$, where $\Psi$ is given by the determinant, and let $\ste = \stw \oplus \stl$, where the subspace $\stw$ has codimension one. Let $0 \neq v \in \ste$ span $\stl$ and let $\nu \in \std$ satisfy $\nu(v) = 1$ and $\ker \nu = \stw$. Equip $\stw$ with the induced affine structure and the parallel volume form $\om = \imt(v)\Psi$ and define the operator $\H$ on $\stw$ with respect to this induced equiaffine structure. Given $f \in \cinf(\stw)$ define $F = \nu - f\circ \pi$, where $\pi:\ste \to \stw$ is the projection along $\stl$.
It is convenient to abuse notation by suppressing $\pi$ and identifying $f$ with its pullback $f \circ \pi$ to $\stw$. The graph $\{(w, f(w)v) \in \stw \oplus \stl\}$ of $f$ along $v$ is contained in the level set $\Sigma = \{x \in \ste: F(x) = 0\}$. Lowercase Latin indices indicate tensors on $\ste$ while uppercase Latin indices indicate tensors on $\stw$. Tensors on $\stl$ are written as functions, without labels. 
Using block matrix notation in a formal way to indicate the decomposition of tensors corresponding to the splitting $\ste = \stw \oplus \stl$,
\begin{align}\label{hessG}
&F_{i} = \begin{pmatrix} -f_{I} & 1 \end{pmatrix},&
&F_{ij} = \begin{pmatrix} -f_{IJ} & 0\\ 0 & 0\end{pmatrix},&
&U^{ij} = \begin{pmatrix} 0 & 0\\ 0 & (-1)^{n}\H(f)\end{pmatrix}.&
\end{align}
From \eqref{hessG} it is apparent that $\H(F) = 0$, $N = U^{ip}F_{p} = (-1)^{n}\H(f)v$, $\U(F) =  (-1)^{n}\H(f)$, and $(n+2)\muf = d\log \H(f)$. Since, by Lemma \ref{nondegenlemma}, $\Sigma$ is nondegenerate at $p \in \Sigma$ if and only if $\U(F)$ is nonzero at $p$, it follows that $\Sigma$ is nondegenerate at $p$ if and only if $\H(f)$ is nonzero at the projection of $p$ onto $\stw$ along $\stl$. In this case, let $f^{IJ}$ be the symmetric bivector on $\stw$ inverse to the Hessian $f_{IJ} = \nabla_{I}df_{J}$ and let $Z^{I}$ be the pushforward of $f^{IQ}d_{Q}\log|\H(f)|^{1/(n+2)}$ via the immersion $\stw \to \ste$ given by $w \to w + f(w)v$. Since $\muf(N) = 0$ and $h_{ij} = (-1)^{n}\H(f)F_{ij}$, it follows from \eqref{nm1} that the equiaffine normal of the graph of $f$ along $v$ has the expression $\nm = -|\H(f)|^{1/(n+2)}(v - Z)$.  

After an equiaffine transformation, $\ste$, $\stw$, $\L$, and the associated connections and volume forms can always be put in the following standard form. Let $\ste = \rea^{n+1}$ and $\Psi = dx^{1}\wedge \dots \wedge dx^{n+1}$, and regard $\rea^{n}$ as the subspace $\stw = \{x \in \ste: x^{n+1} = 0\}$ with the induced connection, also written $\nabla$, and the volume form $\om =dx^{1}\wedge \dots \wedge dx^{n}$. Here $\nu = dx^{n+1}$ and $F(x_{1}, \dots, x_{n+1}) = x_{n+1} - f(x_{1}, \dots, x_{n})$. 
\end{example}

Once a co-orientation has been fixed, let $\nabla$, $k$, and $S$ be the induced connection, metric, and shape operator of the level set of $F$ determined by the co-oriented equiaffine normal. The pseudo-Riemannian metric $k$ is the \textit{equiaffine} metric, $S$ is the \textit{equiaffine shape operator}, and the \textit{equiaffine mean curvature} $\amc$ is the mean curvature with respect to $\nm$. 

Lemma \ref{amclemma} gives explicit expressions for $S$ and $\amc$. The proof of Lemma \ref{amclemma} uses the identity $\hnabla_{p}U^{ip} = 0$. Differentiating $U^{ip}F_{pj} = \H(F)\delta_{j}\,^{i}$ yields $F_{pj}\hnabla_{k}U^{ip} + U^{ip}F_{jkp} = \H(F)_{k}\delta_{j}\,^{i}$. Contracting this in $i$ and $k$ and using $U^{pq}F_{ipq} = \H(F)_{i}$ gives $F_{pj}\hnabla_{q}U^{qp} = 0$. Hence $\H(F)\hnabla_{p}U^{ip} = U^{ij}F_{pj}\hnabla_{q}U^{qp} = 0$, so when $\H(F) \neq 0$ there holds $\hnabla_{p}U^{ip} = 0$. However, in the proof this fact is needed when $\H(F) = 0$. In this generality the claim is a special case of a general identity for the Newton transforms of a symmetric endomorphism proved in Proposition $2.1$ of \cite{Reilly}.

\begin{lemma}\label{amclemma}
Let $F$ be a $\cinf$ function defined on an open subset of $\rea^{n+1}$. Let $\Omega$ be a connected component with nonempty interior of the region on which $\U(F)$ does not vanish. For $r \in \rea$, let $\lc_{r}(F, \Omega) = \{x \in \Omega: F(x) = r\}$. The equiaffine mean curvature $\amc$ of $\lc_{r}(F, \Om)$ is given by
\begin{align}\label{amc}
\begin{split}
n\amc &= \hnabla_{p}(k^{pq}\muf_{q}) + (n+2)|\U(F)|^{1/(n+2)}\U(F)^{-1}(\H(F) - N^{p}\muf_{p})\\
& =  \hnabla_{p}(k^{pq}\muf_{q}) + (-1)^{n+1}(n+2)\sign(\U(F))(\det \hnabla \rho) /\Psi^{2},
\end{split}
\end{align}
where  the \textbf{equiaffine conormal one-form} $\rho_{i} = -|\U(F)|^{-1/(n+2)}F_{i}$ annihilates the tangent space to the level sets of $F$ and satisfies $\rho_{i}W^{i} = 1$.
\end{lemma}

\begin{proof}
Differentiating \eqref{nm1} and using $\hnabla_{i}N^{j} = \H(F)\delta_{i}\,^{j} + F_{p}\hnabla_{i}U^{jp}$ yields
\begin{align}\label{amc2}
\hnabla_{i}\nm^{j} & = -|\U(F)|^{1/(n+2)}\U(F)^{-1}\left(\H(F)\delta_{i}\,^{j} - F_{p}\hnabla_{i}U^{jp} + (n+1)\muf_{i}N^{j} \right) - \hnabla_{i}(k^{jp}\muf_{p}).
\end{align}
Contracting \eqref{amc2} with $\rho_{j}$ and using 
\begin{align}\label{duf}
d_{i}\U(F) = 2\H(F)F_{i} + F_{p}F_{q}\hnabla_{i}U^{pq}
\end{align}
and
\begin{align}
\begin{split}
k^{jp}F_{ip} &= |\U(F)|^{1/(n+2)} \sff^{jp}(\sff_{ip} + \U(F)^{-1}\H(F)F_{j}F_{p}) \\
&= |\U(F)|^{1/(n+2)}(\delta_{j}\,^{i} - \U(F)^{-1}F_{j}N^{i})
\end{split}
\end{align}
yields $\rho_{j}\hnabla_{i}\nm^{j} = \U(F)^{-1}(N^{p}\muf_{p} - \H(F))F_{i}$. Consequently,
\begin{align}\label{equishape}
\begin{split}
S_{i}\,^{j} & = -\hnabla_{i}\nm^{j} + \U(F)^{-1}(N^{p}\muf_{p} - \H(F))F_{i}\nm^{j} \\
& = \hnabla_{i}(k^{jp}\muf_{p})  + |\U(F)|^{1/(n+2)}\U(F)^{-1}\left(\H(F)\delta_{i}\,^{j} - F_{p}\hnabla_{i}U^{jp}\right.\\&\qquad \left. + (n+1)\muf_{i}N^{j}  + (\H(F) - N^{p}\muf_{p})\rho_{i}\nm^{j}\right),
\end{split}
\end{align}
satisfies $S_{i}\,^{p}F_{p} = 0$, so it makes sense to speak of the restriction $S_{I}\,^{J}$ of $S_{i}\,^{j}$ to the tangent bundle of $\lc_{r}(F, \Om)$, and $S_{I}\,^{J}$ is the equiaffine shape operator of $\lc_{r}(F, \Om)$. From \eqref{equishape} it follows that $n\amc = S_{I}\,^{I} = S_{i}\,^{i}$ equals the first expression in \eqref{amc}.  Differentiating $\rho_{i}$ shows
\begin{align}\label{hnablanu}
\hnabla_{i}\rho_{j} = -|\U(F)|^{-1/(n+2)}F_{ij} - \muf_{i}\rho_{j} = -k_{ij} + (\U(F)^{-1}\H(F)F_{i} - \muf_{i})\rho_{j}.
\end{align}
Applying \eqref{rankone} to \eqref{hnablanu} yields
\begin{align}\label{dethrho}
\det \hnabla \rho = (-1)^{n+1}|\U(F)|^{-(n+1)/(n+2)}(\H(F) - N^{p}\muf_{p})\Psi^{2}.
\end{align}
The second equality of \eqref{amc} follows from \eqref{dethrho}.
\end{proof}

Theorem \ref{affinecovariancetheorem} shows that the conditions \eqref{extre} and \eqref{affcov} do not characterize the transversal $\nm^{i}$ because $Q^{i}= -|\U(F)|^{1/(n+2)}\U(F)^{-1}N^{i}$ also satisfies these conditions. By the proof of Lemma \ref{amclemma}, the one forms $F_{p}\hnabla_{i}\nm^{p}$ and $F_{i}$ are proportional. This is not true for $Q^{i}$. Computations using \eqref{sffsff} show that, for any one-form $\si_{i}$, there holds $F_{j}\hnabla^{i}(k^{jp}\si_{p}) = -|\U(F)|^{1/(n+2)}(\si_{i} - \U(F)^{-1}N^{p}\muf_{p}F_{i})$. Consequently if $M^{i} = \nm^{i} + k^{ip}\si_{p}$, then $F_{p}F_{[i}\hnabla_{j]}M^{p} = \rho_{[i}\si_{j]}$, and so $F_{p}F_{[i}\hnabla_{j]}M^{p} = 0$ if and only if $\si_{i}$ is a multiple of $F_{i}$. Since $Q^{i} = \nm^{i} + k^{ip}\muf_{p}$, there holds $F_{p}F_{[i}\hnabla_{j]}Q^{p} = 0$ for all $F \in \uinf(\Om)$ if and only if $\muf \wedge dF = 0$ for all $F \in \uinf(\Om)$. It is straightforward to find $F$ for which this is false, and so the condition that $F_{p}F_{[i}\hnabla_{j]}\A(F)^{p} = 0$ selects $\nm^{i}$ in place of $Q^{i}$.

By \eqref{nmdefined}, $\nm$ and $-|\U(F)|^{1/(n+2)}\U(F)^{-1}N^{i}$ coincide along a level set of $F$ where $\U(F)$ is constant. By the following observation due to R. Reilly, given a level set of $F$ there can always be found a function $G$ such that $G$ and $\U(G)$ are constant along the given level. 

\begin{lemma}[R. Reilly; Proposition $4$ of \cite{Reilly-affine}]
Let $F$ be a $\cinf$ function defined on an open subset of $\rea^{n+1}$. Let $\Omega$ be a connected component with nonempty interior of the region on which $\U(F)$ does not vanish. For $r \in F(\Om)$ let $G = |\U(F)|^{-1/(n+2)}(F - r)$. Then $|\U(G)| = 1$ along $\lc_{r}(F, \Om) = \lc_{0}(G, \Om)$. 
\end{lemma}
\begin{proof}
Let $q \in \cinf(\Om)$ be nonvanishing and define $G = q(F - r)$. By \eqref{ufdet} and elementary determinantal computations
\begin{align}\label{rr1}
\begin{split}
-\U(G) & = \begin{vmatrix} qF_{ij} + 2F_{(i}q_{j)} + (F-r)q_{ij} & qF_{i} + (F-r)q_{i} \\  qF_{j} + (F-r)q_{j} & 0 \end{vmatrix}\\
& = q^{n+2}\begin{vmatrix} F_{ij} - (F - r)q(q^{-1})_{ij}  & F_{i} + (F - r)q^{-1}q_{i}\\ F_{j} + (F - r)q^{-1}q_{j}  & 0 \end{vmatrix}, 
\end{split}
\end{align}
and, again by \eqref{ufdet}, along $\lc_{r}(F, \Om)$ the last expression in \eqref{rr1} equals $-q^{n+2}\U(F)$. Hence, if $q = |\U(F)|^{-1/(n+2)}$, then $\U(G) = \sign(\U(F))$.
\end{proof}
It is not generally possible to find a $G$ having the same level sets as $F$ and having $\U(G)$ constant along each level set. The obstruction is identified in Lemma \ref{uflevellemma}. On the other hand, when it is possible, then there can be found $G$ locally constant on the level sets of $F$ and satisfying $|\U(G)| = 1$.

A version of Lemma \ref{uflevellemma}, with the additional hypothesis that $\H(F) \neq 0$ and slightly weaker conclusions was stated in \cite{Fox-prehom}.

\begin{lemma}\label{uflevellemma}
Let $F$ be a $\cinf$ function defined on an open subset of $\rea^{n+1}$. Let $\Omega$ be a connected component with nonempty interior of the region on which $\U(F)$ does not vanish. For $r \in F(\Om)$, let $\lc_{r}(F, \Omega) = \{x \in \Omega: F(x) = r\}$ and let $\nm$ be its equiaffine normal field \eqref{nm1}. The following are equivalent:
\begin{enumerate}
\item\label{ufc1} $\U(F)$ is locally constant on $\lc_{r}(F, \Om)$ for all $r \in F(\Om)$.
\item\label{ufc2} The equiaffine conormal one-form $\rho$ is closed, $d\rho = 0$.
\item\label{ufc2b} $\muf \wedge \rho = 0$.
\item\label{ufc3} The local flow generated by $\nm$ preserves the level sets of $F$.
\item\label{ufc3b} $\rho$ is constant along the local flow generated by $\nm$.
\item\label{ufc4} $\nm \wedge N = 0$.
\item\label{ufc5} $\nm = -\sign(\U(F))|\U(F)|^{-(n+1)/(n+2)}N$.
\end{enumerate}
If there hold \eqref{ufc1}-\eqref{ufc5} then:
\begin{enumerate}
\setcounter{enumi}{7}
\item The equiaffine mean curvature $\amc$ of $\lc_{r}(F, \Om)$ satisfies
\begin{align}\label{namc}
\begin{split}
n\amc =& (n+2)\sign(\U(F))|\U(F)|^{-(n+1)/(n+2)}(\H(F) - N^{p}\muf_{p}). 
\end{split}
\end{align}
\item \label{grho} If $\Om$ is simply connected, there is $G \in \cinf(\Om)$ satisfying $dG = -\rho$. Then $G$ is locally constant on each level set $\lc_{r}(F, \Om)$ and satisfies $\U(G) = \sign\U(F)$, and so $|\U(G)| = 1$, in $\Om$.
\end{enumerate}
\end{lemma}
\begin{proof}
Skew-symmetrizing \eqref{hnablanu} shows $d\rho = \muf \wedge \rho$. Consequently $\lie_{\nm}\rho = \imt(\nm)d\rho = \muf(\nm)\rho - \muf$, which implies the equivalence of \eqref{ufc2} and \eqref{ufc3b}, and there holds 
\begin{align}\label{dnu}
\rho \wedge \lie_{\nm}\rho = \muf \wedge \rho  = d\rho.
\end{align}
Since the flow of a vector field $X$ preserves the codimension one smooth foliation annihilated by the one-form $\theta$ if and only if $\theta \wedge \lie_{X}\theta = 0$, the equivalence of \eqref{ufc1}-\eqref{ufc3} is immediate from \eqref{dnu}. 
By \eqref{nm1}, $\nm \wedge N = 0$ if and only if $Z =0$. This shows the equivalence of \eqref{ufc4} and \eqref{ufc5}. Since $Z^{p}\sff_{ip} = \U(F)\muf_{i} - N^{p}\muf_{p}F_{i}$, it also shows that $\nm \wedge N = 0$ if and only if $\muf \wedge \rho  = 0$. This shows the equivalence of \eqref{ufc2b} and \eqref{ufc4}. If there hold \eqref{ufc1}-\eqref{ufc5}, then $k^{ip}\muf_{p} = 0$, and \eqref{namc} follows from \eqref{amc}.

If $\Om$ is simply connected, then there exists $G \in \cinf(\Om)$ such that $G_{i} = -\rho_{i} = |\U(F)|^{-1/(n+2)}F_{i}$. Hence 
\begin{align}\label{gij}
G_{ij} = |\U(F)|^{-1/(n+2)}(F_{ij} - F_{i}\muf_{j}) = |\U(F)|^{-1/(n+2)}(F_{ij} - \U(F)^{-1}N^{p}\muf_{p}F_{i}F_{j}),
\end{align}
the second equality by \eqref{ufc2b}. Calculating $\U(G)$ using \eqref{ufdet} and \eqref{gij} and simplifying the result using elementary determinantal operations yields $\U(G) = \sign \U(F)$.
\end{proof}

\begin{corollary}\label{amccorollary}
Let $F$ be a $\cinf$ function defined on an open subset of $\rea^{n+1}$ and suppose $\Omega$ is a connected open set on which $\U(F)$ is equal to a nonzero constant $\ka$. Then the equiaffine normal of the level set $\lc_{r}(F, \Omega)$ is $\nm^{i} = \sign(\ka)|\ka|^{-(n+1)/(n+2)}N^{i}$ and its equiaffine mean curvature $\amc$ satisfies
\begin{align}
\begin{split}
n\amc & = (n+2)\sign(\ka)|\ka|^{-(n+1)/(n+2)}\H(F).
\end{split}
\end{align}
Suppose there is an open domain $\Delta \subset \rea^{n}$ and a smooth map $\phi:\Delta \to \Om$ such that $\phi$ is a diffeomorphism onto its image and $F\circ \phi(x) = t_{0}$ for all $x \in \Delta$ for some $t_{0} \in F(\Om)$. Define $\varphi:\rea \times \Delta \to \Om$ by $\varphi(t, x) = \Phi(-|\ka|^{-1/(n+2)}t, \phi(x))$ where $\Phi(\dum, \dum)$ is the maximal flow of $\nm$ satisfying $\tfrac{d}{dt}\Phi(t, p) = \nm_{\Phi(t, p)}$. Then $\tfrac{d}{dt}\varphi(t, x) = \nm_{\varphi(t, x)}$, so $\varphi$ solves the affine normal flow and $F\circ \varphi(t, x) = t + t_{0}$, so that $\varphi(t,\dum)$ is a parameterization of a subset of $\lc_{t + t_{0}}(F, \Om)$ by an open subset of $\Delta$.
\end{corollary}

\begin{proof}
All the claims except for the final one regarding the affine normal flow follow from specializing Lemma \ref{uflevellemma}. By definition  $\tfrac{d}{dt}\Phi(t, p) = \nm_{\Phi(t, p)}$, so $\tfrac{d}{dt}F\circ \varphi(t, x) = dF(\nm_{\Phi(-|\ka|^{-1/(n+2)}, \phi(x))}) = 1$ for all $(t, x)$ in the domain of definition of $\Phi$. Hence $F\circ \varphi(t, x) = t + t_{0}$.
\end{proof}

Corollary \ref{amccorollary} gives a criterion that is used in Section \ref{examplesection} to check that the examples constructed there have equiaffine mean curvature zero; by Corollary \ref{amccorollary}, if $\U(F)$ is constant on $\Om$ and $\H(F)$ is constant on each level set $\lc_{r}(F, \Om)$, then each level set $\lc_{r}(F, \Om)$ has constant equiaffine mean curvature. In particular, if, on $\Om$, $\H(F)$ vanishes and $\U(F)$ equals a nonzero constant, then each level set $\lc_{r}(F, \Om)$ has equiaffine mean curvature zero.

Lemma \ref{uflevellemma} shows that if $\U(F)$ is nonvanishing and locally constant on the level sets of $F$, then a level set of $F$ is carried into another level set by the affine normal flow equal initially to the affine normal along the first level set. Combining \eqref{grho} of Lemma \ref{uflevellemma} and Corollary \ref{amccorollary} shows that in this case there is a second function $G$, a primitive of a constant multiple of $-\rho$, locally constant on the level sets of $F$, and so that $G(x) - G(y)$ is the the time the affine normal flow takes to move the level set of $F$ containing $y$ to that containing $x$. All these statements are local, but in concrete situations they can make sense globally.

\begin{remark}\label{canonicalpotentialremark}
A nondegenerate hypersurface is an \textit{affine sphere} if its equiaffine shape operator is a multiple of the identity. By the Gauss-Codazzi relations this multiple is necessarily a constant, and the affine sphere is \textit{proper} or \textit{improper} as it is nonzero or zero. 

By a theorem of S.~Y. Cheng and S.~T. Yau, on a proper open convex cone $\Om \subset \rea^{n+1}$ there is a unique $\cinf$ function $F$ solving $\H(F) = e^{2F}$, tending to $\infty$ at the boundary of $\Om$, and such that $F_{ij}$ is a complete Riemannian metric on $\Om$. This $F$ is $-(n+1)$-logarithmically homogeneous and so satisfies $\U(F) = (n+1)\H(F) = (n+1)e^{2F}$. See \cite{Fox-schwarz} for details. Consequently, $F$ and $\Om$ satisfy the conditions of Lemma \ref{uflevellemma}. The function $G$ of \eqref{grho} of Lemma \ref{uflevellemma} is $G = -2^{-1}(n+2)(n+1)^{-1/(n+2)}e^{-2F/(n+2)}$. It satisfies $\U(G) = 1$ and 
\begin{align}
\amc = \tfrac{n+2}{n}\H(G) = -(n+1)^{-(n+1)/(n+2)}e^{2F/(n+2)} = \tfrac{n+2}{2(n+1)}G^{-1},
\end{align} 
so its level sets have constant negative equiaffine mean curvature and are translated by the affine normal flow. In fact the level sets of $G$ (or $F$) are affine spheres, although this does not follow from Lemma \ref{uflevellemma}. The examples constructed in Section \ref{examplesection} show that the level sets of a function $G$ such that $\U(G)$ is constant and $\H(G)$ is constant along the levels of $G$ can have equiaffine mean curvature zero without being improper affine spheres. However, for these examples $\H(G)$ vanishes, and I do not know any example of such a $G$ for which $\H(G)$ is nonvanishing and the level sets of $G$ are not affine spheres. It would be interesting to know if such examples exist. More precisely:
\textit{if the level sets of a function $F$ satisfying $|\U(F)| = 1$ are strictly convex, must they be affine spheres?}
\end{remark}

Example \ref{detexample} illustrates some aspects of the notational conventions used here, as well as some subtleties related to Lemma \ref{unondegenlemma}, in particular to connected components of level sets and the signatures of their equiaffine metrics. As is explained in \cite{Fox-prehom}, the polynomial considered in Example \ref{detexample} is best understood in the general context of prehomogeneous vector spaces, as the fundamental relative invariant of a real form of a prehomogeneous group action.

\begin{example}\label{detexample}
Using the Euclidean metric $\delta_{ij}$ on $\rea^{n}$, identify the space $S^{2}(\rea^{n})$ of symmetric bilinear forms on $\rea^{n}$ with the space $\sm_{n}$ of symmetric $n \times n$ matrices. Endow $\sm_{n}$ with the Riemannian metric (also denoted $\delta$) $\delta(X, Y) = \tr XY$ having Levi-Civita connection $\nabla$ and induced volume form $\Psi$. For $X \in \sm_{n}$, define coordinates $x^{ij}$, $1 \leq i \leq j \leq n$, by $X = \sum_{i = 1}^{n}x^{ii}e_{ii} + \sum_{1 \leq i  < j \leq n}x^{ij}2^{-1/2}(e_{ij} + e_{ji})$ where $e_{ij}$ is the matrix with $ij$ component equal to $1$, and all other components equal to $0$. These coordinates are affine in the sense that the differentials $dx^{ij}$ are $\nabla$-parallel. In them the metric $\delta$ is $\sum_{i \leq j}(dx^{ij})^{2}$ and the volume form satisfies $|\Psi| = |\wedge_{i \leq j}dx^{ij}|$. Were the factor of $2^{-1/2}$ omitted from the definition of the off-diagonal coordinates $x^{ij}$, $i < j$, the resulting coordinates would also be affine, but the expression for $\Psi$ with respect to them would differ by a factor of $2^{\binom{n}{2}}$.
Consider the homogeneous degree $n$ polynomial $P(X) = \det X$ on $\sm_{n}$. With respect to these coordinates, in the $n = 2$ case, $P(X) = x^{11}x^{22} - 2^{-1}(x^{12})^{2}$. This means that the Hessian determinant $\H(P)$ of $P$ is $1$ rather than $2$ as might result from a naive computation ignoring the choice of background volume form and compatible affine coordinates. The rescaling of the off-diagonal elements is an affine transformation of $\sm_{n}$ and the metric properties of $(\sm_{n}, \delta)$ are not directly relevant here; rather they are used to aid computation. In any case, the choice of $\Psi$ is not canonically determined by the choice of $\nabla$, except up to a positive constant factor, and so it is the sign of $\H(P)$, rather than its value, that has affinely invariant significance. 

Let $E_{p} \in \sm_{n}$ be the idempotent diagonal matrix with positive and negative inertia indices $n-2p$ and $2p$.  The level set $\{X \in \sm_{n}: P(X) = 1\}$ is a disjoint union of $\lceil n/2 \rceil$ connected components $\C_{p}$, $0 \leq p < \lceil n/2 \rceil$, where $\C_{p}$ comprises the determinant one matrices congruent to $E_{p}$.
There hold $\tfrac{\pr P}{\pr x^{ij}} = P(X) \tr \left(X^{-1}\tfrac{\pr X}{\pr x^{ij}}\right)$ and
\begin{align}\label{hessdet}
\tfrac{\pr^{2}P}{\pr x^{ij}\pr x^{kl}} = P(X)\left(\tr \left(X^{-1}\tfrac{\pr X}{\pr x^{ij}}\right)\tr \left(X^{-1}\tfrac{\pr X}{\pr x^{kl}}\right) - \tr \left(X^{-1}\tfrac{\pr X}{\pr x^{ij}} X^{-1}\tfrac{\pr X}{\pr x^{kl}}\right)\right).
\end{align}
By the eigenvalues of the Hessian of $P$ are meant the eigenvalues of endomorphism corresponding to it via the Euclidean metric $\delta$.
The eigenvalues of $(\hess P)_{E_{p}}$ can be found from \eqref{hessdet}, using that $\tfrac{\pr X}{\pr x^{ij}}$ equals $e_{ii}$ if $i = j$ and $2^{-1/2}(e_{ij} + e_{ji})$ if $i \neq j$, and that these vectors constitute an orthonormal frame with respect to the metric $\delta$. Say that indices $i$ and $j$ are \textit{separated} if they belong to different subsets of the partition $\{1, \dots, n\} = \{1, \dots, n- 2p\}\cup \{n-2p + 1, \dots, n\}$, and \textit{united} otherwise. At $E_{p}$, the matrix of $\hess P$ with respect to the coordinate frame is block diagonal. Precisely, for $i \leq j$ and $k \leq l$:
\begin{enumerate}
\item\label{detx1} $\tfrac{\pr^{2}P}{\pr x^{ii}\pr x^{jj}}(E_{p})$ is $0$ if $i = j$, $1$ if $i$ and $j$ are united, and $-1$ if $i$ and $j$ are separated; 
\item\label{detx2} if $k$ and $l$ are distinct, then $\tfrac{\pr^{2}P}{\pr x^{ii}\pr x^{kl}}(E_{p}) = 0$;
\item\label{detx3} if $i \neq j$ and $k \neq l$, then $\tfrac{\pr^{2}P}{\pr x^{ij}\pr x^{kl}}(E_{p})$ is $0$ if $ij$ is distinct from $kl$; and $\tfrac{\pr^{2}P}{\pr x^{ij}\pr x^{ij}}(E_{p})$ is equal to $-1$ if $i$ and $j$ are united and $1$ if $i$ and $j$ are separated.  
\end{enumerate}
Let $\eul$ be the Euler vector field generating the dilations $X \to e^{t}X$. Then $dP(\eul) = nP$ and $\hess P(\eul, \dum) = (n-1)dP$. Since the one-form $\delta$-dual to $\eul_{E_{p}}$ is $dP_{E_{p}}$, it follows that $n-1$ is an eigenvalue of $(\hess P)_{E_{p}}$. That the other $n-1$ eigenvalues corresponding to the block $\tfrac{\pr^{2}P}{\pr x^{ii}\pr x^{jj}}(E_{p})$ all equal $-1$ can be seen as follows. Let $I_{k}$ be the $k \times k$ identity matrix, and $\one_{k, l}$ the $k \times l$ matrix all of whose components equal $1$. By \ref{detx1} above, the block $\tfrac{\pr^{2}P}{\pr x^{ii}\pr x^{jj}}(E_{p})$, for $1 \leq i, j \leq n$, has the form
\begin{align}
\begin{pmatrix}
\one_{n-2p, n-2p} - I_{n-2p} & -\one_{n-2p, 2p}\\
-\one_{2p, n-2p}&\one_{2p, 2p} - I_{2p}
\end{pmatrix} = \begin{pmatrix}
I_{n-2p} & 0\\
0& - I_{2p}
\end{pmatrix}  \begin{pmatrix}
\one_{n, n} - I_{n}
\end{pmatrix}\begin{pmatrix}
I_{n-2p} & 0\\
0& - I_{2p}
\end{pmatrix}
\end{align}
and the matrix $\one_{n, n} - I_{n}$ has eigenvalues $n-1$ and $-1$ with multiplicities $1$ and $n-1$. By \ref{detx1}-\ref{detx3} above, the remaining eigenvalues of $\hess P_{E_{p}}$ are $1$, with multiplicity $2p(n-2p)$, and $-1$, with multiplicity $\binom{n-2p}{2} + \binom{2p}{2}$. Hence the positive and negative inertia indices of $(\hess P)_{E_{p}}$ are $2p(n-2p) + 1$ and $\binom{n+1}{2} -1 - 2p(n-2p)$. Additionally, it follows that 
\begin{align}\label{hdet}
\H(P)(E_{p}) = (-1)^{\binom{n+1}{2}-1}(n-1) = (-1)^{n}(n-1)(-P(E_{p}))^{(n+1)(n-2)/2}.
\end{align} 
The connected component of the identity, $G$, of the general linear group of $\rea^{n}$ acts on $\sm_{n}$ linearly by conjugation. Write $\rho: G \to GL(\sm_{n})$ for the induced linear representation. It is straightforward to check that the linear transformation $\rho(g)$ of $\sm_{n}$ induced by $g \in G$ satisfies $\det \rho(g) = (\det g)^{n+1}$ (it suffices to check this for $g$ a multiple of the identity). From \eqref{hgf} of Lemma \ref{hgflemma} and $(\rho(g)\cdot P)(X) = P(\rho(g^{-1})X) = (\det g)^{-2}P(X)$ it follows that
\begin{align}\label{detrel}
\begin{split}
\rho(g)\cdot H(P) &= (\det \rho(g))^{2}\H(\rho(g)\cdot P) \\&= (\det g)^{2(n+1)}\H((\det g)^{-2}P) = (\det g)^{(2-n)(n+1)}H(P).
\end{split}
\end{align}  
By \eqref{detrel}, $\H(P)$ and $P^{(n+1)(n-2)/2}$ are homogeneous polynomials of the same degree transforming identically under the action of $G$ on $\sm_{n}$. As this action has an open orbit containing $E_{p}$, with \eqref{hdet} this implies $\H(P) =  (-1)^{n}(n-1)(-P)^{(n+1)(n-2)/2}$. 

If $F$ has positive homogeneity $\la$ on $\rea^{n}$, then $(\la-1)N^{i} = (\la - 1)U^{ij}F_{j} = U^{ij}F_{jk}\eul^{k} = \H(F)\eul^{i}$ and $(\la -1)\U(F) = (\la - 1)N^{j}F_{j} = \la \H(F) F$, where $\eul$ is the Euler vector field on $\rea^{n}$. Because $P$ has positive homogeneity $n$ on $\sm_{n}$, there results $(n-1)\U(P) = nP\H(P)$ and the transversal $N$ equals $(n-1)^{-1}\H(P)\eul$. Hence $\U(P)$ and $\H(P)$ are nonzero and locally constant on the level set $\{X \in \sm_{n}: P(X) = 1\}$, and so it follows from \eqref{nmdefined} that along this level set the equiaffine normal is a negative constant times $\eul$. 
From \eqref{msff} of Lemma \ref{unondegenlemma} and Lemma \ref{affinenormaldefinitionlemma} it follows that on the level set of $P$ contained in the connected component $\C_{p}$ the positive and negative inertial indices of the equiaffine metric are $2p(n-2p)$ and $\binom{n+1}{2} -1 - 2p(n-2p)$. Note that when $p = 0$ this yields a negative definite metric and, by \eqref{namc}, the level set has positive equiaffine mean curvature. This reflects that, as can be checked using \eqref{hessdet}, $\log P$ is concave on $\C_{0}$. Replacing $P$ on $\C_{0}$ by the function $-P$ reverses the signature of the induced metric and the sign of the equiaffine mean curvature. Since $\U(P)$ is a multiple of $P^{n-1}$, it vanishes exactly on the locus of degenerate matrices in $\sm_{n}$, and a connected component of the subset of $\sm_{n}$ where $\U(P)$ does not vanish equals one of the subspaces comprising matrices of a fixed nondegenerate signature.

Because the equiaffine normal is a multiple of $\eul$, the connected components of the nonzero level sets of $P$ are affine spheres. From \eqref{hpsif} it follows that $\H(-\log P) = P^{-n-1}$. As a consequence, the function $F$ of Remark \ref{canonicalpotentialremark}, that solves $\H(F)  = e^{2F}$ on the convex cone $C_{0}$, is 
$F = -\tfrac{n+1}{2}\left(\log P - \tfrac{n}{2}\log\tfrac{n+1}{2}\right)$.
\end{example}

\subsection{Another definition of the equiaffine normal}\label{affinesection}
Let $\Sigma$ be a co-orientable nondegenerate immersed hypersurface in $\rea^{n+1}$. A transversal $N$ along $\Sigma$ determines on $\Sigma$ the induced volume form $\imt(N)\Psi$ given by interior multiplication with the volume form $\Psi$ on $\rea^{n+1}$, and the volume density $\vol_{h}$ determined by $h$ and the orientation consistent with $\imt(N)\Psi$. By definition $\vol_{h}^{2} = |\det h|$. Next there is given a second definition of a vector field $\nm$ transverse to $\Sigma$ that is manifestly equiaffinely covariant, and it is shown that the transversal so defined coincides with the equiaffine normal as defined in the textbook \cite{Nomizu-Sasaki}, where it is defined up to co-orientation by the requirements that the volume densities $|\imt(\nm)\Psi|$ and $\vol_{h}$ be the same and be parallel with respect to the connection $\nabla$ induced by $\hnabla$ via $\nm$. This characterization of the equiaffine normal is used later to prove Theorem \ref{affinenormalsametheorem}, that shows that the transversal \eqref{nm1} of Theorem \ref{affinecovariancetheorem} is the equiaffine normal.

The approach described now has as a side benefit a somewhat more general construction, namely it attaches to nondegenerate immersed hypersurface in an $(n+1)$-dimensional manifold $M$ equipped with an affine connection $\hnabla$ a distinguished transverse line field that transforms covariantly under the action of the group of affine automorphisms of $(M, \hnabla)$. 
This was described previously in section $4$ of \cite{Fox-ahs}, but is hidden there in a still more general context requiring substantially more preparation, so it is recapitulated here.
 
It is straightforward to check that, on an $n$-manifold, given a pair $(\en, [h])$ comprising a projective structure $\en$ and a conformal class $[h]$ of pseudo-Riemannian metrics there is a unique representative $\nabla \in \en$ that satisfies $nh^{pq}\nabla_{p}h_{qi} = h^{pq}\nabla_{i}h_{pq}$ for every $h \in [h]$. This $\nabla$ is said to be \textit{aligned} with respect to $[h]$. The associated density-valued tensor $H_{ij} = |\det h|^{-1/n}h_{ij}$ is conformally invariant in the sense that it does not depend on the choice of $h \in [h]$. Alternatively, the aligned representative is characterized by the equivalent requirement that there vanish all contractions of $\nabla_{i}H_{jk}$ with the dual density-valued bivector $H^{ij} = |\det h|^{1/n}h^{ij}$. This makes the independence of the choice of $h$ obvious. If $\tnabla = \nabla + 2\ga_{(i}\delta_{j)}\,^{k}$, then $h^{pq}(n\tnabla_{p}h_{qi} - \tnabla_{i}h_{pq}) - h^{pq}(n\nabla_{p}h_{qi} - \nabla_{i}h_{pq}) = (n+2)(1-n)\ga_{i}$. The alignment condition thus determines $\ga_{i}$, and so determines a unique representative of $\en$.

The description of the affine normal to be given now can be summarized as follows. A co-oriented nondegenerate immersed hypersurface $\Sigma$ in a manifold $M$ equipped with an affine connection $\hnabla$ carries the conformal structure induced by the second fundamental form and the co-orientation and the class of affine connections induced via all possible choices of transverse vector fields. The connections induced on $\Sigma$ by $\hnabla$ via transverse vector fields spanning the same line field are the same, so each transverse line field induces on $\Sigma$ a connection $\nabla$. A distinguished transverse line field is determined by the requirement that the induced connection $\nabla$ be aligned with respect to $[h]$. When there is an $\hnabla$-parallalel volume form $\Psi$ on $M$, a distinguished transverse vector field $\nm$ is determined by the requirement $|\imt(N)\Psi| = |\vol_{h}|$. If $\hnabla$ is moreover projectively flat, then $\nabla$ preserves the volume density of the representative $h\in [h]$ corresponding to $\nm$.

\begin{lemma}\label{taulemma}
Let $\Sigma$ be a nondegenerate immersed hypersurface in an $(n+1)$-dimensional manifold $M$ equipped with a projectively flat affine connection $\hnabla$. Let $N$ be a vector field defined in a neighborhood $B$ in $M$ of some point of $\Sigma$ and transverse to $\Sigma$ along $B \cap \Sigma$ and let $\nabla$, $h$, and $\tau$ be the connection, metric, and connection one-form associated with $\hnabla$ and $N$. Then $\nabla_{[I}h_{J]K} = -\tau_{[I}h_{J]K}$.
\end{lemma}
\begin{proof}
Because $\hnabla$ is projectively flat, for vector fields $X$, $Y$, and $Z$ tangent to $\Sigma$, the curvature $\hat{R}(X, Y)Z = ([\hnabla_{X}, \hnabla_{Y}] - \hnabla_{[X, Y]})Z$ is tangent to $\Sigma$. On the other hand, evaluating $\hat{R}(X, Y)Z$ using \eqref{induced} and projecting the result along $N$ shows that $\nabla_{[I}h_{J]K} = -\tau_{[I}h_{J]K}$.
\end{proof}

\begin{theorem}\label{projnormaltheorem}
Let $\Sigma$ be a nondegenerate immersed hypersurface in an $(n+1)$-dimensional manifold $M$ equipped with an affine connection $\hnabla$. 
\begin{enumerate}
\item\label{an1} There is a unique line field $W$ transverse to $\Sigma$ such that the connection $\nabla$ induced on $\Sigma$ by $\hnabla$ and $W$ is aligned with respect to the metric $k$ induced on an open $U \subset \Sigma$ by any nonvanishing section of $W$ over $U$. 
\item\label{an2} If $\Sigma$ is co-oriented and $\Psi$ is a volume form on $M$ preserved by $\hnabla$ then there is a unique co-oriented section $\nm$ of $W$ such that the volume density $|\imt(\nm)\Psi|$ equals the volume density $|\vol_{k}|$ of $k$. 
\item\label{an2b} If $\Sigma$ is co-oriented and $\hnabla$ is projectively flat and preserves the volume form $\Psi$, then the connection $\nabla$ induced on $\Sigma$ by $\hnabla$ via the transversal $\nm$ of \eqref{an2} preserves the induced volume densities $|\imt(\nm)\Psi|$ and $|\vol_{k}|$.
\item\label{an3} If $M = \rea^{n+1}$ and $\hnabla$ and $\Psi$ are the standard flat affine connection and standard volume form then the transversal $\nm$ of \eqref{an2} is the equiaffine normal. 
\end{enumerate}
\end{theorem}

The line field $W$ of \eqref{an1} is the \textit{affine normal distribution of $\Sigma$} (with respect to $\hnabla$). Note that its definition does not suppose that $\Sigma$ is co-orientable.
The transverse vector field $\nm$ of \eqref{an2} is the \textit{equiaffine normal vector field of $\Sigma$} (with respect to $\hnabla$ and $\Psi$).

The Levi-Civita connection of a pseudo-Riemannian metric of constant nonzero curvature is projectively flat but not flat, and so by \eqref{an1}-\eqref{an2b} it makes sense to speak of the affine normal distribution and equiaffine normal of a nondegenerate immersed hypersurface in manifold equipped with a constant curvature metric.

\begin{proof}
Let $U \subset \Sigma$ be an open neighborhood on which there is a smooth vector field $N$ transverse to $\Sigma$. Any other transversal to $\Sigma$ on $U$ has the form $\tilde{N} = a(N + Z)$ for a smooth function $a$ not vanishing on $U$ and a vector field $Z$ tangent to $U$. The second fundamental form $\tilde{h}$, connection $\tnabla$, and connection one-form $\tilde{\tau}$ determined by $\tilde{N}$ and $\hnabla$ are related to $h$, $\nabla$, and $\tau$ by
\begin{align}\label{transform}
&\tilde{h}_{IJ} = a^{-1}h_{IJ},& &\tnabla = \nabla - h_{IJ}Z^{K}, & &\tilde{\tau}_{I} = \tau_{I} + a^{-1}da_{I} + h_{IP}Z^{P}.
\end{align}
It follows from \eqref{transform} that 
\begin{align}\label{transform2}
\begin{split}
\tilde{h}^{PQ}&\tnabla_{I}\tilde{h}_{PQ} = h^{PQ}\nabla_{I}h_{PQ} + 2Z^{P}h_{IP} - n a^{-1}da_{I},\\
\tilde{h}^{PQ}&\tnabla_{P}\tilde{h}_{QI} = h^{PQ}\nabla_{P}h_{QI} + (n+1)Z^{P}h_{IP} - a^{-1}da_{I},
\end{split}
\end{align}
where $h^{IJ}$ and $\tilde{h}^{IJ}$ are the symmetric bivectors inverse to $h_{IJ}$ and $\tilde{h}_{IJ}$, respectively. By \eqref{transform} and \eqref{transform2}, for any constants $\al$, $\be$, and $\ga$, there holds
\begin{align}\label{prenormalize}
\begin{split}
\al \tilde{\tau}_{I} &+ \be \tilde{h}^{PQ}\tnabla_{P}\tilde{h}_{QI} + \ga \tilde{h}^{PQ}\tnabla_{I}\tilde{h}_{PQ} = \al \tau_{I} + \be h^{PQ}\nabla_{P}h_{QI} + \ga h^{PQ}\nabla_{I}h_{PQ} \\
&+ (\al - \be - n \ga)a^{-1}da_{I} + (\al + (n+1)\be + 2\ga)Z^{P}h_{IP}
\end{split}
\end{align}
If $\al = \be + n\ga$, so that \eqref{prenormalize} does not depend on $a$, then 
\begin{align}\label{prenorm}
\begin{split}
\be(\tilde{\tau}_{I} &+  \tilde{h}^{PQ}\tnabla_{P}\tilde{h}_{QI}) + \ga(n\tilde{\tau}_{I} +  \tilde{h}^{PQ}\tnabla_{I}\tilde{h}_{PQ})\\
& = \be(\tau_{I} + h^{PQ}\nabla_{P}h_{QI}) + \ga(n\tau_{I} + h^{PQ}\nabla_{I}h_{PQ}) + (n+2)(\be + \ga)Z^{P}h^{IP}.
\end{split}
\end{align}
Since \eqref{prenorm} does not depend on $a$, as long as $\be \neq -\ga$, the direction of a transversal on $U$ is determined by requiring that the left side of \eqref{prenorm} vanish. The condition $\be \neq -\ga$ is explained as follows. The proof of Lemma \ref{taulemma} applied to the not necessarily projectively flat connection $\hnabla$ shows that $\nabla_{[I}h_{J]K} + \tau_{[I}h_{J]K}$ is expressible in terms of the curvature of $\hnabla$, and tracing this relation shows that $(n\tau_{I} + h^{PQ}\nabla_{I}h_{PQ}) - (\tau_{I} + h^{PQ}\nabla_{P}h_{QI})$ is determined by the curvature of $\hnabla$. Hence there is no freedom to choose the value of this quantity. When $\be \neq -\ga$, the transversal determined by the vanishing of the left side of \eqref{prenorm} depends only on the image of $(\be, \ga)$ in the projective line, so there is one parameter of freedom in the choice of a transverse line field. Among these possible normalizations there is one that is distinguished in that the compatibility condition determining it involves only $\nabla$ and $h$ and does not involve $\tau$; this is the condition given by $\be = -n \ga$. This corresponds to the identity
\begin{align}\label{normalize2}
n\tilde{h}^{PQ}\tnabla_{P}h_{QI}  - \tilde{h}^{PQ}\tnabla_{I}h_{PQ} = nh^{PQ}\nabla_{P}h_{QI} - h^{PQ}\nabla_{I}h_{PQ} + (n+2)(n-1)Z^{P}h_{IP},
\end{align}
and requiring that the left side of \eqref{normalize2} vanish is exactly requiring that $\tnabla$ be aligned with respect to $[h]$. This yields
\begin{align}\label{zdetnonflat}
Z^{P}h_{PI} =  \tfrac{1}{(n+2)(1-n)}\left(nh^{PQ}\nabla_{P}h_{QI} - h^{PQ}\nabla_{I}h_{PQ}\right).
\end{align}
The span $W$ of $\tilde{N}$ is well defined, independently of the remaining freedom, which is the choice of $a$. Since around any point of $\Sigma$ there can be found an open subset $U \subset \Sigma$ on which there is a transversal, and since by the uniqueness just proved the line fields constructed on overlapping neighborhoods agree on the overlaps, the transverse line field $W$ is defined globally on $\Sigma$, even in the case that $\Sigma$ is not co-orientable. If $\Sigma$ is co-orientable then the transversal $\tilde{N} = a(N + Z)$ spanning $W$ can be taken to be globally defined. Since $\det \tilde{h} = a^{-n}\det h$ and $\imt(\tilde{N})\Psi = a \imt(N)\Psi$, $|\vol_{\tilde{h}}/\imt(\tilde{N})\Psi| = |a|^{-(n+2)/2}|\vol_{h}|/|\imt(N)\Psi|$, so requiring that the induced volume density $|\imt(\tilde{N})\Psi|$ equal the volume density $|\vol_{\tilde{h}}|$ of the associated metric determines $a$ up to sign, which is fixed by choosing $a$ so that $\nm$ is co-oriented.

If $\hnabla$ is projectively flat, then, by Lemma \ref{taulemma}, 
\begin{align}\label{projflattau}
&n\tau_{I} + h^{PQ}\nabla_{I}h_{PQ} = \tau_{I} + h^{PQ}\nabla_{P}h_{QI},
\end{align}
and similarly for $\tnabla$, $\tilde{h}$, and $\tilde{\tau}$. With \eqref{projflattau}, the vanishing of the left side of \eqref{normalize2} is equivalent to
\begin{align}\label{zdet0}
&n\tilde{\tau}_{I} +  \tilde{h}^{PQ}\tnabla_{I}\tilde{h}_{PQ} = 0 = \tilde{\tau}_{I} + \tilde{h}^{PQ}\tnabla_{P}\tilde{h}_{QI},
\end{align}
while \eqref{zdetnonflat} becomes
\begin{align}\label{zdet}
 &Z^{P}h_{PI} =  -\tfrac{1}{n+2}\left(n\tau_{I} + h^{PQ}\nabla_{I}h_{PQ}\right).&
\end{align}
Now suppose $\hnabla$ is projectively flat and preserves a volume form $\Psi$. 
Let $k$ be the metric and let $\nabla$ and $\tau$ be the connection and connection one-form associated with the normal $\nm$ determined by \eqref{an2}.
By \eqref{nablavolume} and \eqref{zdet}, 
\begin{align}
\begin{split}
-n\tau_{I}|\vol_{k}| = k^{AB}\nabla_{I}k_{AB}|\vol_{k}| = \nabla_{I}|\vol_{k}| = \nabla_{I}|\imt(\nm)\Psi| = \tau_{I}|\imt(\nm)\Psi| = \tau_{I}|\vol_{k}|.
\end{split}
\end{align}
Hence $ 0 = -n\tau_{I}|\vol_{k}|  = \nabla_{I}|\vol_{k}|$, so that $\tau$ vanishes and $|\vol_{k}|$ is $\nabla$-parallel. 
Now suppose $M = \rea^{n+1}$ and $(\hnabla, \Psi)$ is the standard flat equiaffine structure.
The usual definition of the equiaffine normal given in \cite{Nomizu-Sasaki} is that the volume densities $|\imt(\nm)\Psi|$ and $|\vol_{k}|$ coincide and are preserved by the connection $\nabla$ induced on $\Sigma$ by $\hnabla$ via $\nm$, so \eqref{an3} is immediate from \eqref{an2}.
\end{proof}

Theorem \ref{affinenormalsametheorem} shows that the equiaffinely covariant transversal of Theorem \ref{affinecovariancetheorem} coincides with the equiaffine normal.

\begin{theorem}\label{affinenormalsametheorem}
Let $F$ be a $\cinf$ function defined on an open subset of $\rea^{n+1}$. Let $\Omega$ be a connected component with nonempty interior of the region on which $\U(F)$ does not vanish.
The equiaffine normal of a nonempty level set $\lc_{r}(F, \Omega)$ consistent with the co-orientation determined by $-\sign(\U(F))N^{i}$ is given by the expression \eqref{nmdefined}. 
\end{theorem}
\begin{proof}
Let $h$, $\nabla$, and $\tau$ be respectively the second fundamental form, connection, and connection one-form determined according to \eqref{induced} by the transversal $\gn^{i} = U^{ip}F_{p}$ and $\hnabla$. The most general vector field transverse to the $\lc_{r}(F, \Omega)$ has the form $\nm = a(\gn + Z)$ for a nonvanishing smooth function $a$ and a vector field $Z$ tangent to the $\lc_{r}(F, \Omega)$. The restriction to $\lc_{r}(F, \Om)$ of the tensor $h_{ij} = -\U(F)^{-1}\sff_{ij}$ of \eqref{hn} is the second fundamental form determined by $N^{i}$. By \eqref{duf} there holds $F_{p}\hnabla_{i}N^{p} = d\U(F)_{i} - \H(F)F_{i}$, so that $F_{p}(\hnabla_{i}N^{p} - ((n+2)\muf_{i} - \U(F)^{-1}\H(F)F_{i})N^{j}) = 0$ and therefore the connection one-form determined by $N^{i}$ is the restriction of the one-form $\tau _{i} = (n+2)\muf_{i} - \U(F)^{-1}\H(F)F_{i}$.
This can be written as the identity $\tau_{I} = (n+2)\muf_{I}$, that refers only to the tangential directions. By \eqref{uvolrelation} of Lemma \ref{unondegenlemma}, the volume density $|\vol_{h}|$ induced on $\lc_{r}(F, \Omega)$ by $h$ satisfies $|\vol_{h}| = |\U(F)|^{-(n+1)/2}|\imt(N)\Psi|$ and so, by definition of $\muf$ and \eqref{nablavolume},
\begin{align}\label{mufd}
 h^{AB}\nabla_{I}h_{AB} = 2|\vol_{h}|^{-1}\nabla_{I}|\vol_{h}| = - (n+1)(n+2)\muf_{I} + 2\tau_{I}.
\end{align}
Substituting \eqref{mufd} in \eqref{zdet} and using $\tau_{I} = (n+2)\muf_{I}$ yields $-(n+2)Z^{A}h_{IA} = n\tau_{I} + h^{AB}\nabla_{I}h_{AB} = (n+2)\muf_{I}$, so $Z^{P}h_{PI} = -\muf_{I}$. This is the condition defining the vector field also called $Z^{i}$ in \eqref{nm1}, so the two are the same. By the proof of Theorem \ref{projnormaltheorem}, $|a| = |\vol_{h}/\imt(\gn)\Psi|^{2/(n+2)} = |\U(F)|^{-(n+1)/(n+2)}$. It follows that the equiaffine normal consistent with the co-orientation determined by $-\sign(\U(F))N^{i}$ is $\nm^{i} = -\sign(\U(F))|a|(N^{i} + Z^{i})$, which is \eqref{nm1}. 
\end{proof}

\section{Equiaffine mean curvature zero, \texorpdfstring{$2n$}{}-dimensional hypersurfaces ruled by \texorpdfstring{$n$}{}-dimensional planes}\label{examplesection}
This section describes a general construction of hypersurfaces that generalize the usual helicoids. These are $2n$-dimensional hypersurfaces ruled by $n$-planes and having equiaffine mean curvature zero. This construction directly generalizes one for surfaces given by A. Mart\'inez and F. Milan in \cite{Martinez-Milan}, but the author found it by considering modifications of the homogeneous polynomials constructed by Gordan-N\"other in \cite{Gordan-Noether} (these were mentioned in Example \ref{gnexample}) that while having vanishing Hessian determinant have linearly independent partial derivatives. 

At places in this section it is convenient to use slightly abusive coordinate dependent notation. For clarity, the translation to the abstract index notation used throughout the paper is indicated where confusion could arise. 

\subsection{General construction}\label{ruledexamplesection}
Regard $\rea^{n+1}$ as equipped with the standard flat affine connection $\hnabla$ and the parallel volume form $\Psi$ given by the determinant. The pairing between $\rea^{n+1}$ and the dual vector space $\rea^{n+1\,\ast}$ is written $\lb \dum, \dum \ra:\rea^{n+1} \times\rea^{n+1\,\ast} \to \rea$. The same notation is used for the pairing induced on tensor powers. The space $\rea^{n+1\,\ast}$ is endowed with the dual connection induced by $\hnabla$, also denoted by $\hnabla$, and the parallel $(n+1)$-form $\Psi^{\ast}$ dual to $\Psi$ (meaning that $\lb \Psi, \Psi^{\ast} \ra = 1$). 

Let $M$ be an open smooth connected submanifold of $\rea^{n}$, so $M$ is the nonempty interior of a domain in $\rea^{n}$, or the entirety of $\rea^{n}$. Let $A:M \to \rea^{n+1}$ be a smooth map and fix $Q \in \cinf(M)$. Define a smooth function $F:M \times \rea^{n+1\,\ast} \to \rea$ by 
\begin{align}
F(u, x) =\lb A(u), x\ra + Q(u) 
\end{align}
for $u \in M$ and $x \in \rea^{n+1\,\ast}$. 
Equip $M$ with the flat affine connection $D$ and parallel volume form $\om$ induced by the standard flat connection and volume form on $\rea^{n}$. Equip $M \times \rea^{n+1\,\ast}$ with the (flat) product connection $\nabla$ determined by $D$ and $\hnabla$, and the volume form written, with a slight abuse of notation, as $\Upsilon  = \om \wedge \Psi^{\ast}$. The goal of this section is to obtain conditions on the map $A$ that guarantee that the level sets of $F$ are nondegenerate hypersurfaces with equiaffine mean curvature zero with respect to the equiaffine structure $(\nabla, \Upsilon)$. 

Recall the definition of equiaffine coordinates from the beginning of section \eqref{ufsection} and fix coordinates $x^{1}, \dots, x^{n+1}$ on $\rea^{n+1}$ that are equiaffine with respect to $\hnabla$ and $\Psi$. Let $x_{1}, \dots, x_{n+1}$ be coordinates on $\rea^{n+1\,\ast}$ that are equiaffine with respect to $\hnabla$ and $\Psi^{\ast}$ and dual to the equiaffine coordinates $x^{1}, \dots, x^{n+1}$. This means that the one-forms $dx_{1}, \dots, dx_{n+1}$ on $\rea^{n+1\,\ast}$ constitute a coframe dual to that comprising $dx^{1}, \dots, dx^{n+1}$. Because of this duality, a parallel one-form on $\rea^{n+1\,\ast}$ can be regarded as a parallel vector field on $\rea^{n+1}$, and so it is convenient to write $\pr_{x^{i}}$ for $dx_{i}$. Let $u^{1}, \dots, u^{n}$ be equiaffine coordinates on $(M, D, \om)$. While the coordinates $u^{1}, \dots, u^{n}$ are determined up to automorphisms of $(D, \om)$, in what follows it will be more convenient to think of $D$ and $\om$ as determined by the choice of coordinates. Let $a^{1}(u^{1}, \dots, u^{n}), \dots, a^{n+1}(u^{1}, \dots, u^{n})$ be the components of $A$ with respect to the parallel frame $\pr_{x^{1}}, \dots, \pr_{x^{n+1}}$. In these coordinates $F$ is given by $F(u, x) = \sum_{i = 1}^{n+1}x_{i}a^{i}(u) + Q(u)$.

The Hessian of $F$ is defined with respect to the product connection $\nabla$ and the volume form $\Om$. The differential and Hessian of $F$ are
\begin{align}\label{dhessf}
\begin{split}
dF_{u, x} &= \lb dA_{u}, x \ra + dQ_{u} + \lb A(u), dx\ra = \sum_{i = 1}^{n+1}x_{i}da^{i} + dQ + \sum_{i = 1}^{n+1}a^{i}dx_{i},\\
\hess F_{u, x} &= \lb DdA_{u}, x \ra + DdQ_{u} + 2\lb dA_{u}, dx\ra \\
&= \sum_{i = 1}^{n+1}(da^{i} \tensor dx_{i} + dx_{i} \tensor da^{i})  + \sum_{i = 1}^{n+1}x_{i}Dda^{i} + DdQ.
\end{split}
\end{align}
Let $a^{i}_{I} = \frac{\pr a^{i}}{\pr u^{I}} = da^{i}(\pr_{u^{I}})$ and write $da^{i} = \sum_{I = 1}^{n}a^{i}_{I}du^{I}$ and $A_{I} = dA(\pr_{u^{I}})$. With respect to the chosen equiaffine coordinate frames, $dA$ can be regarded as an $(n+1)\times n$ matrix,
\begin{align}\label{da}
dA = \begin{pmatrix} da^{1} \\ \vdots \\ da^{n+1}\end{pmatrix} = \begin{pmatrix} a^{1}_{1} & \dots & a^{1}_{n}\\ \vdots & \dots & \vdots \\ a^{n+1}_{1} & \dots & a^{n+1}_{n}\end{pmatrix} = \begin{pmatrix} A_{1} & \dots & A_{n}\end{pmatrix}
\end{align}
and the Hessian of $F$ as a matrix with block form
\begin{align}\label{hessfmat}
&\hess F = \begin{pmatrix} \sum_{i = 1}^{n+1}x_{i}Dda^{i} + DdQ & dA^{t} \\ dA & 0 \end{pmatrix},
\end{align}
where the superscript $t$ indicates the matrix transpose.

Define a vector field $V$ on $M \times \rea^{n+1\,\ast}$ by
\begin{align}\label{vdef}
V = \sum_{i = 1}^{n+1}(-1)^{i+1}dA^{(i)}\pr_{x_{i}} 
\end{align}
where $dA^{(i)}$ is the determinant of the $n \times n$ matrix obtained from $dA$ as in \eqref{da} by deleting the $i$th row. A vector field on $\rea^{n+1\,\ast}$ can be regarded as a one-form on $\rea^{n+1}$ and, via this identification, for fixed $u \in M$, $V$ is identified with the one-form which when paired with $X \in \rea^{n+1}$ yields
\begin{align}\label{vdef2}
\lb X, V\ra = \Psi(X, A_{1}, \dots, A_{n}) = \begin{vmatrix} X & A_{1} & \dots & A_{n}\end{vmatrix} = \begin{vmatrix} X & dA \end{vmatrix},
\end{align}
where vertical bars indicate the determinant of a matrix, the determinant is defined relative to the volume form $\Psi$, and the various notations in \eqref{vdef2} are synonymous.

Let $\eul$ be the radial (position) vector field on $\rea^{n+1}$ generating dilations centered at the origin and define the $n$-form $\mu = \imt(\eul)\Psi$. If $X_{1}, \dots, X_{n}$ are vector fields on $\rea^{n+1}$, then
\begin{align}
\mu(X_{1}, \dots, X_{n}) = \Psi(\eul, X_{1}, \dots, X_{n}) = \begin{vmatrix} \eul & X_{1} & \dots & X_{n}\end{vmatrix},
\end{align}
where vertical bars indicate the determinant of a matrix. An immersed codimension one submanifold of $\rea^{n+1}$ is \textit{centroaffine} if it does not contain the origin and is everywhere transverse to $\eul$. An immersion $A:M \to \rea^{n+1}$ of an $n$-manifold is \textit{centroaffine} if $A(M)$ is a centroaffine submanifold. 
Equivalently, the pullback $A^{\ast}(\mu)$ is a volume form on $M$. If $\phi:M^{\prime} \to M$ is a diffeomorphism, then $\phi^{\ast}A^{\ast}(\mu)$ is nonvanishing if and only if $A^{\ast}(\mu)$ is nonvanishing, so the immersion $A$ is centroaffine if and only if the immersion $A \circ \phi:M^{\prime} \to \rea^{n+1}$ is centroaffine.

Define a one-form $\be$ on $M \times \rea^{n+1\,\ast}$ by
\begin{align}\label{bedefined}
\be = \lb A, dx \ra = \sum_{i = 1}^{n+1}a^{i}dx_{i},
\end{align}
and let $\Om = d\be$. Straightforward computations using $da^{1} \wedge \dots \wedge da^{n+1} = 0$ show
\begin{align}\label{becontact}
\be \wedge \Om^{n} = (-1)^{n(n+1)/2}n!A^{\ast}(\mu) \wedge \Psi^{\ast} = (-1)^{n(n+1)/2}n!\begin{vmatrix} A & dA \end{vmatrix} \Upsilon, 
\end{align}
and $\be(V) = \begin{vmatrix} A & dA \end{vmatrix}$. Since $\Om(V, \partial_{u^{I}}) = \begin{vmatrix} A_{I} & dA \end{vmatrix} = 0$, there holds $\imt(V)\Om = 0$. Hence
\begin{align}\label{vom}
\begin{split}
\begin{vmatrix} A & dA \end{vmatrix} \Om^{n} & = \imt(V)(\be \wedge \Om^{n}) = (-1)^{n(n+1)/2}n!\begin{vmatrix} A & dA \end{vmatrix} \imt(V)\Upsilon.
\end{split}
\end{align}

\begin{lemma}\label{immersionlemma}
Let $M$ be an open smooth connected submanifold of $\rea^{n}$ equipped with the induced flat affine connection $D$ and the induced parallel volume form $\om$. Given a smooth map $A:M \to \rea^{n+1}$ and $Q \in \cinf(M)$ define the smooth function $F:M \times \rea^{n+1\,\ast} \to \rea$ by $F(u, x) =\lb A(u), x\ra + Q(u)$ for $u \in M$ and $x \in \rea^{n+1\,\ast}$. Equip $\rea^{n+1\,\ast}$ with the standard flat affine connection $\hnabla$ and parallel volume form $\Psi^{\ast}$, and define the Hessian $\hess F$ with respect to the product flat connection $\nabla$ determined by $D$ and $\hnabla$ and the volume form $\Upsilon = \om \wedge \Psi^{\ast}$. Define a vector field $V$ on $M \times \rea^{n+1\,\ast}$ by \eqref{vdef}. Then:
\begin{enumerate}
\item $V$ is contained in the radical of $\hess F$.
\item\label{uvv} The adjugate tensor $\adj \hess F$ equals $(-1)^{n}V \tensor V$ and the vector field $N^{i}= U^{ij}F_{j}$ equals $N = (-1)^{n}\begin{vmatrix} A & dA \end{vmatrix}V$.
\item $A$ is an immersion if and only if the rank of $\hess F$ is $2n$. In this case the vector field $V$ generates the radical of $\hess F$.
\item\label{cau} The following are equivalent
\begin{enumerate}
\item $A$ is a centroaffine immersion.
\item $\U(F)$ is nowhere vanishing on $M \times \rea^{n+1\,\ast}$. 
\item The one-form $\be$ defined in \eqref{bedefined} is a contact one-form.
\end{enumerate}
In this case, a nontrivial level set $\Sigma$ of $F$ is a smoothly immersed nondegenerate submanifold of $M \times \rea^{n+1\,\ast}$ co-oriented by $N = (-1)^{n}\begin{vmatrix} A & dA \end{vmatrix}V$, and $\adj \hess F = (-1)^{n}\begin{vmatrix} A & dA \end{vmatrix}^{-1}N \tensor N$. Moreover, the restriction to $\Sigma$ of $\Om = d\be$ is a symplectic form.
\item If there is a nonzero constant $\ka$ such that
\begin{align}\label{adak}
A^{\ast}(\mu) = \ka \om
\end{align}
(equivalently, $\be \wedge \Om^{n} = (-1)^{n(n+1)/2}n! \ka \Upsilon$)
then $A$ is a centroaffine immersion and, for a nontrivial level set $\Sigma = \{(u, x) \in M \times \rea^{n+1\,\ast}: F(u, x) = t\}$ of $F$, there hold:
\begin{enumerate}
\item $\Sigma$ has zero equiaffine mean curvature and equiaffine normal
\begin{align}\label{wadak}
\nm = (-1)^{n+1}|\ka|^{-2(n+1)/(n+2)}N = -\sign(\ka)|\ka|^{-n/(n+2)}V.
\end{align}
\item The equiaffine shape operator $S$ of $\Sigma$ is nilpotent with square equal to zero, and $\ker S$ contains the tangent distribution $\T$ of a ruling of $\Sigma$ by $n$-dimensional affine planes Lagrangian with respect to the symplectic form $\Om$.
\item The ruling $\T$ of $\Sigma$ is not cylindrical.
\item The equiaffine metric $h$ of $\Sigma$ has split signature.
\item The Reeb field of $\be$ is $\ka^{-1}V$.
\item\label{param} 
Define $\Phi:\rea \times \rea^{n} \times \rea^{n} \to M \times \rea^{n+1\,\ast}$ by 
\begin{align}\label{levelparam}
\Phi(t, r, s) = (r^{1}, \dots, r^{n}, c^{1}(t, r, s), \dots, c^{n+1}(t, r, s)),
\end{align}
where
\begin{align}
c^{i}(t, r, s) = (-1)^{i+1}\ka^{-1}(t - Q(r))dA^{(i)}(r) + e^{i}(t, r, s),
\end{align}
for $e^{1}(t, r, s) = s^{1}a^{2}(r)$, $e^{n+1}(t, r, s) = -s^{n}a^{n}(r)$, and $e^{i}(t, r, s) = -s^{i-1}a^{i-1}(r) + s^{i}a^{i+1}(r)$ for $2\leq i \leq n$. Then $F(\Phi(t, r, s)) = t$ and for each $t \in \rea$ for which $\Sigma =  \{(u, x) \in M \times \rea^{n+1\,\ast}: F(u, x) = t\}$ is nonempty, $\Phi(t, \dum, \dum): \rea^{n} \times \rea^{n} \to M \times \rea^{n+1,\ast}$ is a parameterization of $\Sigma$.
\item The map $\varphi:\rea \times \rea^{n} \times \rea^{n} \to M \times \rea^{n+1,\ast}$ defined by $\varphi(t, r, s) = \Phi(-|\ka|^{2/(n+2)}t, r, s)$ solves the affine normal flow
\begin{align}\label{mcflow}
\tfrac{d}{dt}\varphi(t, r, s) = \nm_{\varphi(t, r, s)}.
\end{align}
\end{enumerate}
\end{enumerate}
\end{lemma}

\begin{proof}
Since $dF(\pr_{x_{i}}) = a^{i}$, $dF = 0$ if and only if $A$ is the zero map. Hence if $A$ is not the zero map, every level set of $F$ is regular, so smoothly immersed.

Since $\hess F(V, \dum) = \sum_{i = 1}^{n+1}(-1)^{i+1}dA^{(i)}da^{i}$, $\hess F(V, X)$ vanishes for any $X$ tangent to $\rea^{n+1\,\ast}$. On the other hand,
\begin{align}
\hess F(V, \pr_{u^{I}}) = \sum_{i = 1}^{n+1}(-1)^{i+1}dA^{(i)}a^{i}_{I} = \begin{vmatrix} A_{I} & A_{1} & \dots & A_{n} \end{vmatrix} = 0.
\end{align}
Hence  $V$ is contained in the radical of $\hess F$, meaning $\hess F(V, \dum) = 0$, and $\hess F$ is degenerate with rank no greater than $2n$.
 
To analyze the rank and adjugate of $\hess F$ it is most straightforward to examine the matrix representation \eqref{hessfmat}. Consider a $2n \times 2n$ submatrix of \eqref{hessfmat}. If this submatrix is obtained by deleting a row or column not intersecting the null block in the lower right corner, then it contains an $(n+1) \times (n+1)$ null block, and its determinant is $0$. If this submatrix is obtained by deleting a row and column intersecting the null block in the lower right corner, then it contains four $n \times n$ blocks, the lower right of which is null, and its determinant is  the product of the determinants of the antidiagonal blocks multiplied by $(-1)^{n^{2}} = (-1)^{n}$. These antidiagonal blocks are $n \times n$ submatrices of $dA$ and its transpose and so the corresponding minor equals $(-1)^{n}dA^{(i)}dA^{(j)}$. Since the entry in the row $n+j$ and column $n+i$ of the matrix representing $\adj \hess F$ is this minor multiplied by $(-1)^{(n+i)+ (n+j)} = (-1)^{i+j}$, it equals $(-1)^{i+j+n}dA^{(i)}dA^{(j)}$. This shows that 
\begin{align}\label{adjhessf}
\adj \hess F = \begin{pmatrix}
0 & 0 \\ 
0 & C
\end{pmatrix} = (-1)^{n}V \tensor V 
\end{align}
where the $(n+1)\times (n+1)$ matrix $C$ has components $C_{ij} = (-1)^{i+j+n}dA^{(i)}dA^{(j)}$. 
(In the abstract index notation used in most of this paper, \eqref{adjhessf} would be written $U^{ij} = (-1)^{n}V^{i}V^{j}$.) 

If $u^{1}, \dots, u^{n}$ are equiaffine coordinates on $(M, D, \om)$, then $A^{\ast}(\mu) = \begin{vmatrix} A & dA \end{vmatrix}du^{1} \wedge \dots \wedge du^{n}$, or, equivalently,
\begin{align}
\begin{vmatrix} A & dA \end{vmatrix}= \begin{vmatrix} A & A_{1} & \dots & A_{n} \end{vmatrix} = \Psi(A, A_{1}, \dots, A_{n}) =  A^{\ast}(\mu)(\pr_{u^{1}}, \dots, \pr_{u^{n}}),
\end{align}
so that, from \eqref{vdef2}, it follows that $dF(V) = \begin{vmatrix} A & dA \end{vmatrix}$. Consequently, multiplying \eqref{adjhessf} by $dF$ yields
\begin{align}
\label{nav}
&N = (-1)^{n}dF(V)V = (-1)^{n}\begin{vmatrix} A & dA \end{vmatrix}V,\\
\label{exuf}
&\U(F) = dF(N) = (-1)^{n}dF(V)^{2} = (-1)^{n}\begin{vmatrix} A & dA \end{vmatrix}^{2}.
\end{align}

That $A$ be an immersion means that $dA$ has rank $n$ everywhere. The preceding shows directly that the rank of $dA$ is $n$ if and only if $\hess F$ has a nonvanishing $2n \times 2n$ minor, so has rank $2n$. Precisely, $V$ is nowhere vanishing if and only if $dA$ has rank $n$ everywhere, and because $\adj \hess F = (-1)^{n}V \tensor V$ this holds if and only if $\adj \hess F$ has rank $1$ everywhere, which is equivalent to $\hess F$ having rank $2n$ everywhere.

Consequently, if $A$ is an immersion, then the radical of $\hess F$ has dimension one, and since the radical contains the nowhere vanishing vector field $V$, it must be that $V$ generates the radical, meaning that any vector $X$ satisfying $\hess F(X, \dum) = 0$ is a multiple of $V$.

By \eqref{exuf}, $\U(F)$ is nowhere vanishing if and only if $A^{\ast}(\mu)$ is nowhere vanishing, so $A$ is a centroaffine immersion if and only if $\U(F)$ is nowhere vanishing. By \eqref{becontact} this holds if and only if $\be$ is a contact one-form. In this case, by Lemma \ref{unondegenlemma}, there hold the claims in \eqref{cau} regarding $\Sigma$ and $N$. Together \eqref{nav} and \eqref{adjhessf} show $\adj \hess F = (-1)^{n}\begin{vmatrix} A & dA \end{vmatrix}^{-1}N \tensor N$. By \eqref{vom}, $\Om^{n} = (-1)^{n(n+1)/2}n!\imt(V)\Upsilon$, and so the restriction of $\Om$ to $\Sigma$ is nondegenerate. This shows \eqref{cau}.

If there is a constant $\ka \neq 0$ such that there holds \eqref{adak}, then, by \eqref{exuf}, $\U(F) = \ka^{2}$ is constant, so, by Corollary \ref{amccorollary}, the level sets of $F$ have equiaffine mean curvature zero and equiaffine normal given by \eqref{wadak}. It is apparent from the explicit form \eqref{vdef} of $V$ that the equiaffine shape operator $S$ is nilpotent. More precisely, its square is $0$, for its image is contained in the span of the coordinate vector fields $\pr_{x_{1}}, \dots, \pr_{x_{n+1}}$, while its kernel contains these vector fields. Note that the equiaffine normals are constrained to lie in a linear subspace. 

For each $u \in M$ let $c_{1}, \dots, c_{n}$ be vectors in $\rea^{n+1}$ spanning the kernel of the one-form $A(u)$. The components $c_{Ii}$ of $c_{I}$ with respect to the equiaffine coordinates $x^{i}$ on $\rea^{n+1}$ are functions of $u$ alone. The linearly independent vector fields $T_{I} = \sum_{p = 1}^{n+1}c_{Ip}\pr_{x_{p}}$ satisfy $dF(T_{I}) = 0$, so are tangent to $\Sigma$. As $\nabla_{T_{I}}T_{J} = 0$, the $T_{I}$ span a totally geodesic rank $n$ distribution $\T$ tangent to $\Sigma$; its leaves are $n$-dimensional affine planes. Since $\nm$ is a constant multiple of $V$ and $\nabla_{T_{I}}V = 0$, $\T \subset  \ker S$. The annihilator $\ann \T$ of $\T$ is spanned by $\be$ and $du^{1}, \dots, du^{n}$. Were $\T$ cylindrical, then $\ann \T$ would be preserved by $\hnabla$. Since the $du^{i}$ are $\hnabla$-parallel, this can be only if there are smooth functions $b_{I}(u)$, $1 \leq I \leq n$, such that $\hnabla \be = \sum_{I = 1}^{n}b_{I}du^{I}\tensor \be$. Equivalently, $\frac{\pr a^{i}}{\pr u^{I}} = b_{I}a^{i}$. Differentiating this identity yields $\tfrac{\pr^{2}a^{i}}{\pr u^{I}\pr u^{J}} = b_{I}b_{J}a^{i} + a^{i}\tfrac{\pr b_{I}}{\pr u^{J}}$. Hence $a^{i}(\tfrac{\pr b_{I}}{\pr u^{J}} - \tfrac{\pr b_{J}}{\pr u^{I}}) = 0$. Were $a^{i}$ to vanish on an open set, then on this open set there would vanish $\begin{vmatrix} A & dA \end{vmatrix}$, a contradiction. If $\bar{u}$ is a point where no component of $A(\bar{u})$ vanishes, then the same holds in an open neighborhood $U$ of $\bar{u}$. Hence on $U$ there holds $\tfrac{\pr b_{I}}{\pr u^{J}} = \tfrac{\pr b_{J}}{\pr u^{I}}$, so on some possibly smaller neighborhood of $\bar{u}$, also to be called $U$, there is a smooth function $f$ such that $b_{I} = \tfrac{\pr f}{\pr u^{I}}$. Then $\tfrac{\pr}{\pr u^{I}}(e^{-f}a^{i}) = 0$, so on $U$ there are constants $c^{i} \neq 0$ such that $a^{i} = c^{i}e^{-f}$. However, this implies $\begin{vmatrix} A & dA \end{vmatrix} = e^{-(n+1)f}\begin{vmatrix}c & -b_{1}c & \dots & -b_{n}c \end{vmatrix} = 0$, which is a contradiction. It follows that $\T$ is not cylindrical. For any vector field $X$ tangent to $\Sigma$, $\nabla_{X}V$ is contained in the span of the vector fields $\pr_{x_{1}}, \dots, \pr_{x_{n+1}}$. Since $V$ is also contained in this span, and $\nm$ is a constant multiple of $V$, $S(X)$ is contained in the span of the vector fields $\pr_{x_{1}}, \dots, \pr_{x_{n+1}}$. Since then $\nabla_{S(X)}V = 0$, the square of $S$ is $0$. Since $\nabla_{T_{I}}T_{J} = 0$, the equiaffine metric $h$ satisfies $h(T_{I}, T_{J}) = 0$, so $\T$ is an $n$-dimensional $h$-isotropic subspace, which shows that $h$ has split signature. Since $\Om(T_{I}, T_{J}) = 0$, $\T$ is Lagrangian. Since $\be(V) = \ka$ and $\imt(V)\Om = 0$, the Reeb field of $\be$ is $\ka^{-1}V$.

That \eqref{levelparam} satisfies $F(\Phi(t, r, s)) = t$ follows from the observation that $\sum_{i = 1}^{n+1}a^{i}(r)c^{i}(t, r, s) = t$. To show \eqref{mcflow}, it is helpful to rewrite the parameterization \eqref{levelparam} of $\Sigma$ of \eqref{param} in the following way. Let $U_{I} = \pr_{u^{I}}$ and $X_{i} = \pr_{X^{i}}$. For $1 \leq I \leq n$ define $Y_{I} = a^{I+1}X_{I} - a^{I}X_{I+1}$. Then 
\begin{align}\label{param2}
\begin{split}
\Phi(t, r, s) &= \sum_{I = 1}^{n}r^{I}U_{I} + \sum_{I = 1}^{n}s^{I}Y_{I}(r) + \ka^{-1}(t - Q(r))V \\
&= \sum_{I = 1}^{n}r^{I}U_{I} + \sum_{I = 1}^{n}s^{I}Y_{I}(r) + |\ka|^{-2/(n+2)}(Q(r) - t)\nm .
\end{split}
\end{align}
From \eqref{param2} it is clear that $\tfrac{d}{dt}\Phi(t, r, s) = -|\ka|^{-2/(n+2)}\nm_{\Phi(t, r, s)}$, from which \eqref{mcflow} follows.
\end{proof}

\begin{remark}
If $A$ is an immersion and $A(u) = 0$ then the rank of $\hess F$ is $2n$ at $(u, x)$ although $\U(F)$ vanishes at $(u, x)$. This does not contradict Lemma \ref{nondegenlemma} because at $(u, x)$ the nontrivial vector field $V$ is tangent to the level set of $F$ through $(u, x)$ and contained in the radical of $\hess F$, and so this level set is degenerate. Note also that such an immersion $A$ is not centroaffine.
\end{remark}

\begin{remark}
The parameterization \eqref{levelparam} of the level set $\Sigma$ of \eqref{param} or \eqref{param2} of Lemma \ref{immersionlemma} can be rewritten as
\begin{align}\label{param3}
\Phi(t, r, s) = \begin{pmatrix} 0 \\ C(r) \end{pmatrix}s + \begin{pmatrix} r \\ D(t, r)\end{pmatrix},
\end{align}
where the $(n+1) \times n$ matrix $C(r)$ and the $(n+1)$-vector $D(t, r)$ have components
\begin{align}
&C_{iJ}(r) = \begin{cases} a^{i+1}(r) & \text{if}\,\,i = J,\\ -a^{i}(r)&  \text{if}\,\, i = J+1,\\ 0 & \text{otherwise}, \end{cases}&& D^{i}(t, r) = (-1)^{i+1}\ka^{-1}(t - Q(r))dA^{(i)}(r),
\end{align}
where $1 \leq i \leq n+1$ and $1 \leq J \leq n$. From the representation \eqref{param3} it is clear that $\Sigma$ is ruled by $n$-planes.
\end{remark}

\begin{lemma}\label{affinespherelemma}
Let $M$ be an open smooth connected submanifold of $\rea^{n}$ equipped with the induced flat affine connection $D$ and the induced parallel volume form $\om$. Equip $\rea^{n+1\,\ast}$ with the standard flat affine connection $\hnabla$ and parallel volume form $\Psi^{\ast}$. Suppose $Q \in \cinf(M)$ and let $A:M \to \rea^{n+1}$ be a centroaffine immersion satisfying $A^{\ast}(\mu) = \ka \om$ for a nonzero constant $\ka$. A nonempty connected component $\Sigma$ of a level set of the function $F:M \times \rea^{n+1\,\ast} \to \rea$ defined by $F(u, x) =\lb A(u), x\ra + Q(u)$ is an affine sphere if and only if the image of $A$ is contained in a hyperplane.
\end{lemma}

\begin{proof}
Since, by Lemma \ref{immersionlemma}, the shape operator of $\Sigma$ is nilpotent, if $\Sigma$ is an affine sphere it is necessarily improper. This means that the equiaffine normal $\nm$ is parallel along the hypersurface. Since $\nm$ is a constant multiple of the vector field $V$ defined in \eqref{vdef}, it follows that $V$ is $\hnabla$-parallel along $\Sigma$, and so the components $(-1)^{i+1}dA^{(i)}$ of $V$ are equal to constants $p_{i}$, that can be viewed as the components of a constant vector $p \in \rea^{n+1\,\ast}$. Then $\ka = \begin{vmatrix} A & dA \end{vmatrix} = \lb A(u), p\ra$, so the image of the centroaffine immersion $A$ is contained in the hyperplane $\{v \in \rea^{n+1}: \ka = \lb v, p\ra\}$. 

Now suppose $A$ is a centroaffine immersion satisfying $\begin{vmatrix} A & dA \end{vmatrix} = \ka$ for a nonzero constant $\ka$ and having image contained in the hyperplane $\{v \in \rea^{n+1}: 1 = \lb v, p\ra\}$ for a constant vector $p \in \rea^{n+1\,\ast}$. First suppose that $p = dx^{n+1}$. Then $a^{n+1} = 1$, so 
\begin{align}\label{damin}
&dA = \begin{pmatrix}
a^{1}_{1} & \dots & a^{1}_{n}\\ \vdots & \dots & \vdots \\ a^{n}_{1} & \dots & a^{n}_{n}\\ 0 & \dots & 0 
\end{pmatrix},& &\ka  = \begin{vmatrix} A & dA \end{vmatrix} = \begin{vmatrix} a^{1} & a^{1}_{1} & \dots & a^{1}_{n}\\ \vdots & \vdots & \dots & \vdots \\a^{n} & a^{n}_{1} & \dots & a^{n}_{n}\\ 1 & 0 & \dots & 0 \end{vmatrix} = (-1)^{n}dA^{(n+1)}.
\end{align}
Hence, for $1 \leq i \leq n$, the minors $dA^{(i)}$ are zero, while $dA^{(n+1)}$ is constant.
This shows that $V$ is a constant vector, so parallel, so $\Sigma$ is an improper affine sphere. For general $p$ there is $g \in Gl(n+1, \rea)$ such that $g^{\ast}p = dx^{n+1}$, where $g^{\ast}$ indicates the adjoint action defined by $\lb u, g^{\ast}mu \ra = \lb g^{-1}u, \mu\ra$ for $v \in \rea^{n+1}$ and $\mu \in \rea^{n+1\,\ast}$. Then $1 = \lb A, p \ra = \lb gA, dx^{n+1}\ra$. Since $\begin{vmatrix} gA & d(gA) \end{vmatrix} =(\det g)\begin{vmatrix} A & dA \end{vmatrix} = \ka \det g$, the preceding shows that $d(gA)^{(i)} = 0$ for $1 \leq i \leq n$ and $d(gA)^{(n+1)} = (-1)^{n}\ka \det(g)$. By the Cauchy-Binet formula for minors, for $1 \leq i \leq n+1$, $d(gA)^{(i)} = \sum_{j = 1}^{n+1}|g^{(ij)}|dA^{(j)}$, where $|g^{(ij)}|$ is the determinant of the $n \times n$ submatrix of $g$ obtained by deleting the $i$th row and $j$th column. This shows that the minors $dA^{(1)}, \dots, dA^{(n+1)}$ solve $n+1$ constant coefficient linear equations, and so they must be constants. Hence $V$ is parallel and $\Sigma$ is an affine sphere.
\end{proof}

In order to obtain interesting examples from Lemma \ref{immersionlemma}, it is necessary to solve the partial differential equation \eqref{adak} for some constant $\ka \neq 0$. Let $\phi:M \to M^{\prime}$ be a diffeomorphisms between open domains in $\rea^{n}$. Then
\begin{align}\label{adachange}
|d\phi| \begin{vmatrix} A\circ \phi^{-1} & d(A\circ \phi^{-1}) \end{vmatrix} \circ \phi = \begin{vmatrix} A & dA \end{vmatrix}.
\end{align}
If $\phi$ is to be such that $\begin{vmatrix} A\circ \phi^{-1} & d(A\circ \phi^{-1}) \end{vmatrix} = \ka$ for some constant $\ka \neq 0$, then it follows from \eqref{adachange} that it must be that
\begin{align}\label{adaeqn}
|d\phi| = \ka^{-1} \begin{vmatrix} A & dA \end{vmatrix}.
\end{align}
Solving \eqref{adaeqn} is simply the problem of finding a diffeomorphism that pulls a given volume form back to a standard volume form. Beginning with J. Moser's \cite{Moser}, solutions to several versions of this problem have been obtained, for various combinations of hypotheses regarding regularity, total volume, and the topology of the underlying spaces. 

\begin{lemma}\label{deformationlemma}
Let $\Sigma \subset \rea^{n+1}$ be a centroaffine smoothly immersed codimension one submanifold for which there exists a smooth diffeomorphic parameterization $A:M \to \Sigma$ for some open connected smooth submanifold $M$ of $\rea^{n}$, let $\mu = \imt(\eul)\Psi$, and let $\om$ be the standard volume form on $\rea^{n}$ that is preserved by the standard flat affine connection on $\rea^{n}$. Then:
\begin{enumerate}
\item\label{banyagaclaim} If the closure of $M$ is a compact connected smooth manifold with boundary with respect to the smooth structure induced from $\rea^{n}$, then there is a diffeomorphism $\phi:M \to M$ isotopic to the identity and extending to the identity on the boundary of the closure of $M$ such that $B = A \circ \phi:M \to \rea^{n+1}$ satisfies $B^{\ast}(\mu) = \ka \om$.
\item\label{schlenkclaim} If $M$ has infinite volume with respect to $\om$, there exists a smooth embedding $\phi:\rea^{n} \to M$ such that $B = A \circ \phi:M \to \rea^{n+1}$ satisfies $B^{\ast}(\mu) = \ka \om$ for some nonzero constant $\ka$.
\end{enumerate}
\end{lemma}

\begin{proof}
In the setting of \eqref{banyagaclaim}, by the main theorem of A. Banyaga's \cite{Banyaga-formes-volumes} there exists a diffeomorphism $\phi:M \to M$ isotopic to the identity and extending to the identity on the boundary of the closure of $M$ such that $\phi^{\ast}A^{\ast}(\mu) = \vol_{A^{\ast}(\mu)}(\Sigma)\om$.

In the setting of \eqref{schlenkclaim}, by a theorem of F. Schlenk (see Appendix B of \cite{Schlenk-book}) there exists a smooth embedding $\phi:\rea^{n} \to M$ such that $\phi^{\ast}A^{\ast}(\mu) = \om$.
\end{proof}
A version of claim \eqref{banyagaclaim} with lower regularity assumptions could be obtained by using the more well-known theorem of Dacorogna-Moser, \cite{Dacorogna-Moser}, instead of the cited theorem of Banyaga. 

Lemma \ref{deformationlemma} means that if a centroaffine immersed submanifold of $\rea^{n+1}$ admits a smooth parameterization by an open connected submanifold of $\rea^{n}$ that either has infinite volume with respect to the standard volume form $\om$ on $\rea^{n}$ or has compact closure with infinitely smooth boundary, then it admits a parameterization $A$ by such a submanifold $M$ such that $A^{\ast}(\mu)$ is a nonzero constant multiple of the restriction to $M$ of $\om$. Combining Lemmas \ref{immersionlemma} and \ref{deformationlemma} proves the following theorem.

\begin{theorem}\label{lifttheorem}
Given an immersed hypersurface $\Sigma$ in $\rea^{n+1}$ everywhere transverse to the radial vector field $\eul$ and diffeomorphic to a connected open submanifold $M$ of $\rea^{n}$ that either has infinite volume with respect to the standard volume form $\om$ on $\rea^{n}$ or has compact closure with infinitely smooth boundary, there exists a centroaffine parameterization $A:M \to \Sigma$ satisfying $A^{\ast}(\mu) = \ka \om$ for a nonzero constant $\ka$ and the restriction of $\om$ to $M$, and, for any $Q \in \cinf(M)$, the level sets of the function $F:M \times \rea^{n+1\,\ast} \to \rea$ defined by $F(u, x) =\lb A(u), x\ra + Q(u)$ are equiaffine mean curvature zero smooth submanifolds of $M \times \rea^{n+1}$  ruled by $n$-planes tangent to the kernel of $S$ and for which the equiaffine shape operator $S$ is nilpotent of order two. The one-form $\be = \lb A, dx \ra = \sum_{i = 1}^{n+1}a^{i}dx_{i}$ is a contact one-form on $M \times \rea^{n+1\,\ast}$ for which the Reeb field is a constant multiple of the equiaffine normal of the level sets of $F$, and the restriction of $d\be$ to each level set is a symplectic form with respect to which the ruling of the level set is a Lagrangian foliation.
\end{theorem}

\subsection{Examples}\label{zamcsection}
When $n = 1$, the surfaces obtained from the construction of section \ref{ruledexamplesection} are a subset of those described in section two of \cite{Martinez-Milan}. 

Let $a(t)$ and $b(t)$ be the components of an immersion $A:I \subset \rea \to \rea^{2}$. Then $\begin{vmatrix} A & dA \end{vmatrix} = ab^{\prime} - ba^{\prime}$ equals a nonzero constant $\ka$ if and only if $a$ and $b$ are linearly independent solutions of a homogeneous linear second order differential equation $f^{\prime\prime}(t) + r(t)f(t) = 0$. Suppose that there exist linearly independent twice differentiable solutions $a(t)$ and $b(t)$ defined for all $t \in \rea$. Then $\ka = ab^{\prime} - a^{\prime}b$. Let $Q \in \cinf(\rea)$ and define $F(u, x, y) = a(u)x + b(u)y + Q(u)$. Then $dF = adx + bdy + (a^{\prime}(u)x + b^{\prime}(u)y + Q^{\prime}(u))du$ and $\hnabla dF = (a^{\prime}(u)dx  + b^{\prime}(u)dy) \tensor du + du \tensor (a^{\prime}(u)dx  + b^{\prime}(u)dy) + (xa^{\prime\prime}(u) + yb^{\prime\prime}(u) + Q^{\prime\prime}(u))du \tensor du$. While $\H(F) = 0$, $\U(F) = -(ab^{\prime} - a^{\prime}b)^{2} = -\ka^{2}$, so that $\muf = 0$. The equiaffine normal \eqref{nm1} is $\nm = \sign(\ka)|\ka|^{-1/2}(a^{\prime}(u)\pr_{y} - b^{\prime}(u)\pr_{x})$. The level set $\Sigma_{t} = \{(u, x, y) \in \rea^{3}: F(u, x, y) = t\}$ is the ruled surface given parametrically by 
\begin{align}\label{helicoid}
\begin{split}
\Sigma_{t} &= \{(r, sb(r) + \ka^{-1}(t - Q(r))b^{\prime}(r), -sa(r) - \ka^{-1}(t-Q(r))a^{\prime}(r))\in \rea^{3}: (r, s) \in \rea^{2}\}.
\end{split}
\end{align}
By Corollary \ref{amccorollary}, a linear rescaling of \eqref{helicoid} solves the affine normal flow.
Since $a$ and $b$ do not vanish simultaneously, $Y = b\pr_{x} - a\pr_{y}$ never vanishes and is tangent to $\Sigma$. Since along $\Sigma$ there holds $0 = F(u, x, y) = a(u)x + b(u) y$ there exist functions $p(u)$ and $q(u)$ such that $qx = pb$ and $qy = -pa$. Note that $p$ and $q$ are necessarily nonvanishing. Then the vector field $X = p(b^{\prime}\pr_{x} - a^{\prime}\pr_{y}) + q \pr_{z}$ is tangent to $\Sigma$ and independent of $Y$, for $Y \wedge X = p(ab^{\prime} - a^{\prime}b)\pr_{x}\wedge \pr_{y} = \ka p \pr_{x}\wedge \pr_{y} \neq 0$. Since $\hnabla_{Y}\nm = 0$ and $\hnabla_{X}\nm = \sign(\ka)|\ka|^{1/2}qr Y$, the equiaffine shape operator is given by $S(Y) = 0$ and $S(X) = -\sign(\ka)|\ka|^{1/2}qr Y$, so is nilpotent and has zero trace. This shows directly that the surface $\Sigma$ has equiaffine mean curvature zero. 

Particular special cases of this construction are well known. Taking $q = 1$ and $a(t) = \sin t$ and $b(t) = -\cos(t)$, so $F(u, x, y) = x\sin u - y \cos u$, there results the usual helicoid $\{(r, s\cos r, s\sin r)\in \rea^{3}: (r, s) \in \rea^{2}\}$. In this case $\ka = -1$ and $\nm = \sin u \pr_{x} - \cos u \pr_{y}$. Taking $q = -1$ and $a(t) = e^{-t}$ and $b = -e^{t}$, so $F(u, x, y) = xe^{-u} - ye^{u}$, there results the surface $x = ye^{2u}$. In this case $\ka = 2$ and $\nm = 2^{-1/2}(e^{u}\pr_{x} - e^{-u}\pr_{y})$. These two examples are given in section $3$ of \cite{Verstraelen-Vrancken}.


A modification of these examples allows the construction of examples in higher dimensions.
\begin{lemma}\label{extensionlemma}
Let $M$ be a an open smooth connected submanifold of $\rea^{n-1}$ equipped with the induced flat affine connection and parallel volume form $\om$ and let $B:M \to \rea^{n}$ be a smooth centroaffine immersion satisfying $\begin{vmatrix} B & dB \end{vmatrix} = \ka$ for $0 \neq \ka \in \rea$, with respect to the standard flat affine connection and parallel volume form on $\rea^{n}$. Let $a, b, c \in \cinf(I)$ be smooth functions on a connected open subinterval $I \subset \rea$. Endow $I \times M$ with the product flat affine connection and the parallel volume form $dt \wedge \om$, and regard $\rea^{n+1}$ as $\rea^{n} \times \rea$, endowed with the standard flat affine connection and parallel volume form. If there is a nonzero constant $\bar{\ka}$ such that 
\begin{align}\label{barkay}
(-1)^{n}\ka \left(ab^{\prime} - a^{\prime}a\right)a^{n-1}c^{n-1} = \bar{\ka},
\end{align}
then the map $A:I \times M \to \rea^{n+1}$ defined by $A(t, u) = (a(t) B(uc(t)), b(t))$ is a smooth centroaffine immersion satisfying $\begin{vmatrix} A & dA \end{vmatrix} = \bar{\ka}$. In particular,
\begin{enumerate}
\item\label{barkay1} If $a, b \in \cinf(I)$ are linearly independent solutions on $I \subset \rea$ of a second order linear differential equation and $a$ does not vanish on $I$, then  the map $A:I \times M \to \rea^{n+1}$ defined by $A(t, u) = (a(t) B(u/a(t)), b(t))$ is a smooth centroaffine immersion satisfying $\begin{vmatrix} A & dA \end{vmatrix} = (-1)^{n}\tau\ka$ where $\tau$ is the constant such that $\tau = ab^{\prime} - a^{\prime}b$. 
\item\label{barkay2} If $n$ is even, and $a$ and $ab^{\prime} - a^{\prime}b$ are nonvanishing on $I$, then $c = a^{-1}(ab^{\prime} - a^{\prime}b)^{1/(n-1)}$ yields a solution of \eqref{barkay} with $\bar{\ka} = \ka$.
\item If $n$ is odd, and $a$ and $ab^{\prime} - a^{\prime}b$ are nonvanishing on $I$, then $c = a^{-1}|ab^{\prime} - a^{\prime}b|^{1/(n-1)}$ yields a solution of \eqref{barkay} with $\bar{\ka} = -\sign(ab^{\prime} - a^{\prime}b)\ka$.
\end{enumerate}
\end{lemma}
\begin{proof}
Although this is a straightforward computation, the details are given for clarity. 
Using a slightly abusive matrix notation, and calculating the determinant by expansion by cofactors along the last row, one obtains
\begin{align}
\begin{split}
\begin{vmatrix} A& dA \end{vmatrix} & = 
\begin{vmatrix} 
aB& a^{\prime}B + ac^{\prime}\sum_{i = 1}^{n-1}u^{i}\tfrac{\pr B}{\pr u^{i}} & ac dB\\
b & b^{\prime} & 0 \\
\end{vmatrix} \\
& = (-1)^{n+1}b \begin{vmatrix} a^{\prime} B  + ac^{\prime}\sum_{i = 1}^{n-1}u^{i}\tfrac{\pr B}{\pr u^{i}}& ac dB \end{vmatrix} + (-1)^{n}b^{\prime}\begin{vmatrix} aB & ac dB\\ \end{vmatrix}\\
& = (-1)^{n+1}\left(b \begin{vmatrix} a^{\prime} B & ac dB \end{vmatrix} +b \begin{vmatrix}  ac^{\prime}\sum_{i = 1}^{n-1}u^{i}\tfrac{\pr B}{\pr u^{i}}& ac dB \end{vmatrix} - b^{\prime}\begin{vmatrix} aB & ac dB\\ \end{vmatrix}\right)\\
& = (-1)^{n+1}\left(ba^{\prime} - b^{\prime}a\right)a^{n-1}c^{n-1}\begin{vmatrix} B & dB \end{vmatrix} = (-1)^{n}\ka \left(ab^{\prime} - a^{\prime}a\right)b^{n-1}c^{n-1},
\end{split}
\end{align}
where the penultimate equality holds because $\sum_{i = 1}^{n-1}u^{i}\tfrac{\pr B}{\pr u^{i}}$ is a linear combination of the columns of $dB$. Choosing $c(t) = 1/a(t)$ yields the conclusion \eqref{barkay1}. The remaining conclusions are straightforward.
\end{proof}

For example, applying Lemma \ref{extensionlemma} with $B$ being the helicoid and $a(t) = \cosh t$ and $b(t) = \sinh(t)$ yields, after relabeling variables,
\begin{align}
A(x_{1}, x_{2}) = \begin{pmatrix} \cosh{x_{1}} \cos( x_{2}\sech{x_{1}}) \\ \cosh{x_{1}}\sin( x_{2}\sech{x_{1}})  \\ \sinh{x_{1}} \end{pmatrix},
\end{align}
that satisfies $\begin{vmatrix} A & dA \end{vmatrix} = -1$ on all of $\rea^{2}$. It follows that, for any $Q \in \cinf(\rea^{2})$, the level sets of the function 
\begin{align}
\begin{split}
F(x_{1}, &x_{2}, x_{3}, x_{4}, x_{5}) \\
&= x_{3} \cosh{x_{1}} \cos( x_{2}\sech{x_{1}}) + x_{4} \cosh{x_{1}} \sin( x_{2}\sech{x_{1}}) + x_{5}\sinh{x_{1}} + Q(x_{1}, x_{2})
\end{split}
\end{align}
are equiaffine mean curvature zero smooth hypersurfaces ruled by $2$-planes.

A different example, directly generalizing the helicoid, can be obtained by starting from a centroaffine immersion and explicitly solving for a diffeomorphism as in Theorem \ref{lifttheorem}. Look for an immersion $A: \rea^{2} \to \rea^{3}$ that takes values in the Euclidean unit sphere. The standard parameterization of the upper half sphere by spherical polar coordinates does not satisfy that $\begin{vmatrix} A & dA\end{vmatrix}$ is constant, but by solving \eqref{adaeqn} there is obtained the parameterization $A:\Sigma \to \rea^{3}$,
\begin{align}
A(x_{1}, x_{2}) = \begin{pmatrix} \sqrt{1 - x_{1}^{2}}\cos x_{2} \\ \sqrt{1 - x_{1}^{2}}\sin x_{2} \\ x_{1} \end{pmatrix},
\end{align}
defined on the strip $\Sigma = \{x \in \rea^{2}: x_{1} \in (-1, 1)\}$ and satisfying $\begin{vmatrix} A & dA \end{vmatrix} = -1$. There are many diffeomorphisms from $\rea^{2}$ to $\Sigma$ with Jacobian determinant $1$. Simply take a diffeomorphism $\psi:(-1, 1) \to \rea$ and define $(x_{1}, x_{2}) = \phi(t, s) = (\psi(t),  s/\psi^{\prime}(t))$. For a concrete example, take $\psi(t) = \tanh{t}$. After relabeling $t$ and $s$, there results the parameterization $A:\rea^{2} \to \rea^{3}$,
\begin{align}\label{ahelic}
A(x_{1}, x_{2}) = \begin{pmatrix} \sech{x_{1}} \cos( x_{2}\cosh^{2}{x_{1}}) \\ \sech{x_{1}} \sin( x_{2}\cosh^{2}{x_{1}})  \\ \tanh{x_{1}} \end{pmatrix},
\end{align}
that satisfies $\begin{vmatrix} A & dA \end{vmatrix} = -1$. The immersion \eqref{ahelic} can be obtained from \eqref{barkay2} of Lemma \ref{extensionlemma}, by taking $a = \sech t$, $b = \tanh t$ and $c = \cosh^{2}t$.
It follows that the level sets of the function \eqref{genhel} are equiaffine mean curvature zero smooth hypersurfaces ruled by $2$-planes. 

The special case of Theorem \ref{lifttheorem} where the image $A(M)$ is a graph can be modeled by taking $a^{i+1}(u^{1}, \dots, u^{n}) = u^{i}$ for $1 \leq i \leq n$. Write $a^{1}(u)= f(u)$. Then $dF(V) = f - \sum_{i = 1}^{n}u^{i}f_{i}$ where $f_{i} = df(\pr_{u^{i}})$. Choose any function $g$ so that $\sum_{i= 1}^{n}u^{i}g_{i} = g$, choose $\ka \neq 0$, and define $f = g + \ka$. Then $dF(V) = \ka$, so $\U(F) = \ka^{2}$. The level sets of $F$ passing through points of $\rea^{5}$ in the complement of the set where $g$ fails to be smooth are nondegenerate smooth hypersurfaces having zero equiaffine mean curvature.

To have a concrete example, let $k$ be a positive integer and let $g(u, v) = (u^{2k} + v^{2k})^{1/2k}$. After relabeling variables there results
\begin{align}
F(x_{1}, x_{2}, x_{3}, x_{4}, x_{5}) = (x_{1}^{2k} + x_{2}^{2k})^{1/2k}x_{3} + x_{3} + x_{1}x_{4} + x_{2} x_{5}.
\end{align}
Note that the level set $\{x \in \rea^{5}:F(x) = c\}$ is the graph $\{x_{3} = u(x_{1}, x_{2}, x_{4}, x_{5})\}$ of the function 
\begin{align}
u(x_{1}, x_{2}, x_{4}, x_{5}) = \tfrac{c - x_{1}x_{4} - x_{2}x_{5}}{1 + (x_{1}^{2k} + x_{2}^{2k})^{1/2k}},
\end{align}
on $\{x \in \rea^{5}: x_{3} = 0\}$ along $\pr_{3}$. These level sets are smooth except where both $x_{1}$ and $x_{2}$ vanish. These examples are somewhat unsatisfying because of the nonregularity along the locus $\{x \in \rea^{5}: x_{1} = 0 = x_{2}\}$. 

\section{Comparison of equiaffine and unit normals of a hypersurface in a space form}\label{comparisonsection}
The second fundamental form of an immersion $\Sigma \to M$ with respect to an affine connection $\hnabla$ on $M$ depends only on the the projective equivalence class of $\hnabla$, as, for vector fields $X$ and $Y$ tangent to $\Sigma$, the projections of $\hnabla_{X}Y + \ga(X)Y + Y\ga(X)$ and $\hnabla_{X}Y$ onto the normal bundle of $\Sigma$ are equal. Thus it makes sense to speak of the second fundamental form of an immersion in a manifold equipped with a projective structure, and it makes sense to say that such an immersion is nondegenerate if the second fundamental form is nondegenerate.

From \eqref{induced} it is apparent that the connections induced on $\Sigma$ by a fixed $\hnabla \in [\hnabla]$ with respect to transversals $N$ and $\tilde{N} = aN$ are the same, while the connections $\nabla$ and $\tnabla$ induced on $\Sigma$ by $\hnabla$ and $\hnabla + 2\ga_{(i}\delta_{j)}\,^{k}$ with respect to a fixed transversal $N$ are projectively equivalent, related by $\tnabla = \nabla + 2\ga_{(I}\delta_{J)}\,^{K}$. Here, as throughout this section, tensors on a hypersurface $\Sigma$ are labeled with capital Latin abstract indices. Consequently, $[\hnabla]$ and a line field transverse to $\Sigma$ determine on $\Sigma$ a projective structure $\en$.

Although the nondegeneracy of a hypersurface depends only on the projective equivalence class of the ambient connection, the affine normal distributions induced by different connections generating the same flat projective structure are different. The precise relation is given by Lemma \ref{projchangelemma}.

\begin{lemma}\label{projchangelemma}
Let $\Sigma$ be a nondegenerate immersed hypersurface in an $(n+1)$-dimensional manifold $M$ equipped with a projectively flat affine connection $\hnabla$. If $\nm$ is a local section over the open set $U \subset \Sigma$ of the affine normal distribution determined by $\hnabla$ then the affine normal distribution determined by the projectively equivalent connection $\hnabla + 2\si_{(i}\delta_{j)}\,^{k}$ is spanned over $U$ by $\nm + \si^{\sharp}$ where $\si^{\sharp}$ is any vector field equal to $h^{IP}\si_{P}$ along $U$, where $h$ is the representative of the second fundamental form determined by $\hnabla$ and $\nm$.
\end{lemma}
\begin{proof}
Let $h$, $\tau$, and $\bnabla$ be determined by $\hnabla +  2\si_{(i}\delta_{j)}\,^{k}$ and $\nm$. Let $\nabla$ be the connection induced by $\hnabla$ and $\nm$. Then $\bnabla = \nabla + 2\si_{(I}\delta_{J)}\,^{K}$ and $\tau_{I} = \si_{I}$. As $\nabla$ preserves $\vol_{h}$, $h^{PQ}\bnabla_{I}h_{PQ} = -2(n+1)\si_{I}$. By construction the affine normal distribution determined by $\hnabla +  2\si_{(i}\delta_{j)}\,^{k}$ is spanned by $\tilde{\nm} = \nm + Z$ where $-(n+2)Z^{P}h_{IP} = \tau_{I} + h^{PQ}\bnabla_{I}h_{PQ} = -(n+2)\si_{I}$. This proves the claim. 
\end{proof}
\begin{remark}
In Lemma \ref{projchangelemma}, the hypothesis that $\hnabla$ be projectively flat is unnecessary; see \cite{Fox-ahs}.
\end{remark}

Lemma \ref{projchangelemma} shows that any transversal spans the affine normal distribution determined by some connection projectively equivalent to a given projectively flat connection. So, while there is a model of hyperbolic space for which the Levi-Civita connection of the hyperbolic metric is projectively equivalent to the standard flat Euclidean connection, the affine normals of a hypersurface determined by these connections are different. This raises the possibility that there are interesting examples in the affine geometry of hypersurfaces in constant curvature pseudo-Riemannian spaces. 

The Levi-Civita connection $\hnabla$ of a pseudo-Riemannian manifold $(M, g)$ is projectively flat if and only if $g$ has constant curvature. When the metric $g$ has indefinite signature there are two notions of nondegeneracy of an immersed co-oriented hypersurface $\Sigma \subset M$. First, there is the equiaffine notion of nondegeneracy with respect to $\hnabla$, necessary to make sense of the co-oriented equiaffine normal $\nm$ determined by $\hnabla$ and and the volume form $\vol_{g}$ of $g$, as in Theorem \ref{projnormaltheorem}. Second, there is the usual pseudo-Riemannian notion of nondegeneracy of $\Sigma$, that the restriction of $g$ to $\Sigma$ be nondegenerate, necessary to make sense of the co-oriented unimodular normal field $E$ $g$-orthogonal to $\Sigma$. That $E$ be unimodular means that $|g(E, E)| = 1$. Theorem \ref{normalstheorem} shows that for a hypersurface nondegenerate in both senses these normals are proportional exactly when the hypersurface has constant Gauss-Kronecker curvature. By the Gauss-Kronecker curvature is meant the determinant of the shape operator determined by the Levi-Civita connection of $g$ and the co-oriented unimodular normal $E$.

\begin{theorem}\label{normalstheorem}
Let $(M, g)$ be a constant curvature pseudo-Riemannian manifold, and let $\Sigma$ be a co-oriented immersed hypersurface in $M$ such that $\Sigma$ is nondegenerate with respect to the the Levi-Civita connection $\hnabla$ of $g$ and the restriction of $g$ to $\Sigma$ is a pseudo-Riemannian metric. Then the co-oriented equiaffine normal $\nm$ of $\Sigma$ determined by $\hnabla$ and the volume form $\vol_{g}$ of $g$ and the co-oriented unimodular orthogonal vector field $E$ are proportional if and only if $\Sigma$ has constant Gauss-Kronecker curvature, in which case $\nm$ is a constant multiple of $E$ and the equiaffine mean curvature of $\Sigma$ is a constant multiple of the pseudo-Riemannian mean curvature.
\end{theorem}

\begin{proof}
Let $D$, $A$, and $\Lambda$ be the induced connection, shape operator, and second fundamental form determined by $\hnabla$ with respect to the co-oriented unimodular vector field $E$ $g$-orthogonal to $\Sigma$. By definition, the Gauss-Kronecker curvature of $\Sigma$ is $\det A$. The connection one-form determined by $\hnabla$ and $E$ is zero, $D$ is the Levi-Civita connection of the restriction of $g$ to $\Sigma$, and $\Lambda$ and $A$ are related by $g(AX, Y) = \ep\Lambda(X, Y)$ for $X$ and $Y$ tangent to $\Sigma$ and $\ep = g(E, E)$. By the proof of Theorem \ref{projnormaltheorem}, $\nm = a(E + Z)$ where, by \eqref{zdet}, $-(n+2)Z^{P}\Lambda_{IP} = \Lambda^{PQ}D_{I}\Lambda_{PQ} = |\det \Lambda|^{-1}D_{I}|\det \Lambda| = D_{I}\log \det A$ (the last equality because $D$ preserves $g$) and $a^{n+2} = |\vol_{\Lambda}/(\imt(E)\vol_{g})|^{2} = |\det A|$. Since  $Z$ is tangent to $\Sigma$, $E$ is orthogonal to $\Sigma$ and nonvanishing, and $a$ does not vanish because $\Sigma$ is assumed nondegenerate, there vanishes $\nm \wedge E = aZ \wedge E$ if and only if $Z = 0$. Since $-(n+2)Z^{P}\Lambda_{IP} = D_{I}\log \det A$, this occurs if and only if $\det A$ is constant. 
\end{proof}

By Theorem \ref{normalstheorem} a nondegenerate hypersurface in a pseudo-Riemannian space form that has constant pseudo-Riemannian mean curvature zero and nonzero constant Gauss-Kronecker curvature has constant equiaffine mean curvature. If the space form is three-dimensional then these conditions mean that the pseudo-Riemannian principal curvatures of the hypersurface are constant. In \cite{Chen-minimal}, B.-Y. Chen showed that if $M$ is a surface with Riemannian mean curvature zero and constant Gauss curvature in a three-dimensional curvature $c$ Riemannian space form, then either $M$ is totally geodesic or $c > 0$ and $M$ is isometric to an open submanifold of a Clifford torus. The Clifford torus is a special case of the following well-known example of a minimal embedding with constant Gauss-Kronecker curvature. Let $\sphere^{n}(r)$ be the radius $r$ sphere in $(n+1)$-dimensional Euclidean space. In \cite{Chern-Docarmo-Kobayashi} it is shown that the embedding 
\begin{align}\label{spheresphere}
\sphere^{m}\left(\sqrt{m/n}\right) \times \sphere^{n-m}\left(\sqrt{(n-m)/n}\right) \to \sphere^{n+1}(1)
\end{align} 
induced by the embedding $\rea^{m+1} \times \rea^{n-m+1} \to \rea^{n+2}$ is minimal, with $m$ principal curvatures equal to $\pm \sqrt{(n-m)/m}$ and $n-m$ principal curvatures equal to $\mp\sqrt{m/(n-m)}$. In particular, this embedding has constant Gauss-Kronecker curvature. This shows that there are nontrivial examples of equiaffine mean curvature zero hypersurfaces in spheres. Recall that a hypersurface in a space form is \textit{isoparameteric} if its pseudo-Riemannian principal curvatures are constant. Since such a hypersurface has constant Gauss-Kronecker curvature, if no pseudo-Riemannian principal curvature is zero, Theorem \ref{normalstheorem} applies to show that its equiaffine mean curvature is constant. 
\begin{corollary}
In a pseudo-Riemannian space form, the equiaffine mean curvature of an isoparametric hypersurface having nonzero Gauss-Kronecker curvature is constant.
\end{corollary}
It is reasonable to ask under what conditions a constant equiaffine mean curvature hypersurface in a pseudo-Riemannian space form must be isoparametric. The question is very general and here no attempt is made to review the related literature. The answer must depend on the curvature of the ambient space form, as is illustrated by the theorem of B.-Y. Chen mentioned above. 

\bibliographystyle{amsplain}

\def\cprime{$'$} \def\cprime{$'$} \def\cprime{$'$} \def\cprime{$'$}
  \def\cprime{$'$} \def\cprime{$'$}
  \def\polhk#1{\setbox0=\hbox{#1}{\ooalign{\hidewidth
  \lower1.5ex\hbox{`}\hidewidth\crcr\unhbox0}}} \def\cprime{$'$}
  \def\Dbar{\leavevmode\lower.6ex\hbox to 0pt{\hskip-.23ex \accent"16\hss}D}
  \def\cprime{$'$} \def\cprime{$'$} \def\cprime{$'$} \def\cprime{$'$}
  \def\cprime{$'$} \def\cprime{$'$} \def\cprime{$'$} \def\cprime{$'$}
  \def\cprime{$'$} \def\cprime{$'$} \def\cprime{$'$} \def\cprime{$'$}
  \def\dbar{\leavevmode\hbox to 0pt{\hskip.2ex \accent"16\hss}d}
  \def\cprime{$'$} \def\cprime{$'$} \def\cprime{$'$} \def\cprime{$'$}
  \def\cprime{$'$} \def\cprime{$'$} \def\cprime{$'$} \def\cprime{$'$}
  \def\cprime{$'$} \def\cprime{$'$} \def\cprime{$'$} \def\cprime{$'$}
  \def\cprime{$'$} \def\cprime{$'$} \def\cprime{$'$} \def\cprime{$'$}
  \def\cprime{$'$} \def\cprime{$'$} \def\cprime{$'$} \def\cprime{$'$}
  \def\cprime{$'$} \def\cprime{$'$} \def\cprime{$'$} \def\cprime{$'$}
  \def\cprime{$'$} \def\cprime{$'$} \def\cprime{$'$} \def\cprime{$'$}
  \def\cprime{$'$} \def\cprime{$'$} \def\cprime{$'$} \def\cprime{$'$}
  \def\cprime{$'$} \def\cprime{$'$} \def\cprime{$'$} \def\cprime{$'$}
\providecommand{\bysame}{\leavevmode\hbox to3em{\hrulefill}\thinspace}
\providecommand{\MR}{\relax\ifhmode\unskip\space\fi MR }
\providecommand{\MRhref}[2]{%
  \href{http://www.ams.org/mathscinet-getitem?mr=#1}{#2}
}
\providecommand{\href}[2]{#2}

\end{document}